%% file: top-fb.tex
\documentclass[12pt,reqno]{amsart}
\usepackage{amsmath, amssymb, esint}
\usepackage[margin=1in]{geometry}
\usepackage{amsthm}
\usepackage{fullpage,amsbsy,url,enumerate,comment,colonequals}
\usepackage{parskip}

\usepackage{bbm}

\usepackage[pdftex]{graphicx,color}
\pdfoutput=1

\newtheorem{definition}{Definition}[section]
\newtheorem{theorem}{Theorem}[section]
\newtheorem{prop}[theorem]{Proposition}
\newtheorem{lemma}[theorem]{Lemma}

\newtheorem{conjecture}[theorem]{Conjecture}

\newtheorem{remark}[theorem]{Remark}
\newtheorem{question}[theorem]{Question}

\def\R{\mathbb{R}}
\def\C{\mathbb{C}}
\def\Z{\mathbb{Z}}

\def\e{\epsilon}
\def\G{\Gamma}
\def\D{\Delta}
\def\DD{\mathbb{D}}
\def\O{\Omega}

\def\a{\alpha}
\def\b{\beta}
\def\g{\gamma}
\def\d{\delta}

\def\th{\theta}
\def\de{\partial}

\def\k{\kappa}
\def\z{\zeta}
\def\disp{\displaystyle}

\def\noi{\noindent}

\def\Re{\mbox{Re} \,}
\def\Im{\mbox{Im} \,}
\newcommand{\chrc}[1]{\mathrm{1}_{#1}}

\def\Xint#1{\mathchoice
{\XXint\displaystyle\textstyle{#1}}%
{\XXint\textstyle\scriptstyle{#1}}%
{\XXint\scriptstyle\scriptscriptstyle{#1}}%
{\XXint\scriptscriptstyle\scriptscriptstyle{#1}}%
\!\int}
\def\XXint#1#2#3{{\setbox0=\hbox{$#1{#2#3}{\int}$ }
\vcenter{\hbox{$#2#3$ }}\kern-.6\wd0}}

\def\dashint{\Xint-}

\def\F{\Phi}

\def\B{\mathbb{B}}

\def\HH{\mathcal{H}} 
\def\CC{\mathcal{C}} 

\def\S{\mathcal{S}}
\def\Dee{\mathcal{D}}

\newcommand{\eps}{\epsilon}

\title{Free Boundaries Subject to Topological Constraints}
\date{}
\author{David Jerison and Nikola Kamburov}
\address{David Jerison, Department of Mathematics,  Room 2-272,
Massachusetts Institute of Technology, Cambridge, MA 02139, USA}
\email{jerison@math.mit.edu}
\address{Nikola Kamburov, Facultad de Matem\'aticas, Pontificia Universidad Cat\'olica de Chile, Avenida Vicu\~na Mackenna 4860, Santiago 7820436, Chile}
\email{nikamburov@mat.uc.cl}

\thanks{Supported in part by NSF Grant DMS 1500771, a Simons Fellowship, and Simons Foundation grant 
(601948, DJ). NK was partially supported by Proyecto FONDECYT Iniciaci\'on No. 11160981}


\numberwithin{equation}{section}

\begin{document}

\begin{abstract}  We discuss the extent to which solutions to one-phase
free boundary problems can be characterized according to
their topological complexity.    Our questions are motivated by  fundamental
work of Luis Caffarelli on free boundaries and by striking results
of T. Colding and W. Minicozzi concerning finitely connected, embedded, minimal surfaces.
We review our earlier work on the simplest case,
one-phase free boundaries in the plane in which
the positive phase is simply connected.  We also prove a new, purely topological,
effective removable singularities theorem for free boundaries.  At the same time, we 
formulate some open problems concerning the multiply connected case and 
make connections with the theory of minimal surfaces and semilinear variational problems.
\end{abstract}

\bibliographystyle{alpha}

\maketitle

\section{Introduction}  This paper is dedicated to Luis Caffarelli on his 70th birthday.
Over the course of more than four decades, Luis
 has shaped  the theory of elliptic nonlinear partial differential equations.  
 In particular, he has had a profound influence on both of us.  With this offering, we express our deep appreciation for his wisdom, guidance, generosity, and
friendship.

Consider a bounded, open set $\O\subset \R^n$.   In 1981, Alt and Caffarelli \cite{AC} showed that minimizers $u$ of the functional 
\begin{equation}\label{functional}
    J(v,\O) = \int_\O (|\nabla v|^2 + 1_{\{v>0\}}) ~dx,  \qquad
    v :\O\to[0,\infty).
\end{equation}
with fixed boundary conditions on $\de \O$ are Lipschitz continuous
in $\O$ and satisfy Euler-Lagrange (one-phase free boundary) equations
\begin{equation}\label{FBP}
 \Delta u   = 0  \ \text{in} \ \O^+(u), \quad     |\nabla u|  = 1 \ \text{on} \  F(u)
 \end{equation}
 in a suitable generalized sense, with positive phase $\O^+(u)$ 
 and free boundary $F(u)$ given by 
\[
\O^+(u) : = \{x\in \O: u(x)>0\}; \quad F(u) = \O \cap \de \O^+(u).
\]
In particular, they proved that the free boundary $F(u)$ is a smooth hypersurface
except on a relatively closed subset of $(n-1)$-Hausdorff measure
zero, $u$ is smooth up to $F(u)$ from the positive side in 
a neighborhood of every point where $F(u)$ is smooth,
and the equation $|\nabla u| = 1$ is satisfied at those points.
The paper \cite{AC} revealed a broad analogy with
the theory of minimal surfaces, which continues to guide further 
progress.

The focus of the present paper is  families of 
one-phase solutions (i.~e. non-negative solutions) to the Euler-Lagrange
equation of the Alt-Caffarelli functional that are not necessarily 
minimizers.    We made partial progress on the problem by characterizing
the interior behavior of such solutions in the 2-dimensional
case in which the positive phase is simply connected \cite{JK}.  
Here we explain those results somewhat more informally and prove a new, fully topological
version of one of them, an effective removable singularities theorem.
We also formulate several questions that are still open in dimension two,
especially concerning the multiply connected case and the case of solutions 
obtained by the Perron process.  We explain the analogy with questions about 
minimal surfaces and more general questions concerning solutions to 
semilinear  equations in dimensions two and three.

\begin{definition} \label{def:classical-soln}  A classical solution to the
one-phase free boundary problem in an open set  $\O\subset \R^n$ 
is a continuous, non-negative function $u$ defined on $\O$ 
such that \mbox{$u\in C^\infty(\bar \O^+(u))$,} 
the free boundary $F(u)$ is a $C^\infty$ hypersurface with $u=0$
on one side and $u>0$ on the other, and $u$ satisfies \eqref{FBP} 
with the gradient $\nabla u$ evaluated from the positive side. 
\end{definition}

Let $B_r(x)$ denote the open ball centered at $x$ of radius $r$ in $\R^n$,
and set $B_r = B_r(0)$, the ball centered at the origin.  
Let $\DD$ denote the unit disk centered at the origin in $\R^2$. 
Consider a classical solution $u$ to the one-phase free boundary problem
in $\DD$.     The main question we will discuss is what 
$\DD^+(u)$ and $F(u)$ look like within the smaller disk $B_{1/2}$.

One of the most important methods used to study free boundary problems
is to take rescaled limits, both blow-up and blow-down solutions,
an approach closely
associated with Luis Caffarelli in many contexts.
In \cite{ACF, CafI, CafII} the authors developed a systematic
theory of two-phase free boundary problems, that is, the analogous
variational problem with solutions that are allowed to change sign.  
In the last of his series of articles
on the Harnack inequality approach \cite{CafIII}, Caffarelli found
a different way to construct solutions, namely by taking the infimum
of strict supersolutions.  At the same time, he developed the
appropriate notion of weak solution that is closed under uniform
limits.  These solutions are now called viscosity solutions.
As valuable as they are for the study of minimizers of the functional,
they are even more essential in the study of one-phase non-minimizers.

The method of rescaling often reduces questions about more
general solutions to the characterization of entire solutions.   In the plane,
the obvious solutions are the single (half-) plane solution $P$ and 
two-plane solutions $TP_a$, $a>0$, defined as functions of $z = x_1 + ix_2$ by 
\begin{equation} \label{eq:planesolns}
P(z) = x_1^+; \quad TP_a(z) := x_1^+ + (-x_1 - a)^+, \qquad a>0.
\end{equation}
and rigid motions of these solutions.   Note that under rescaling by the
factor $a>0$, we have $P(z) =  a \, P(z/a)$, $TP_a= a \, TP_1(z/a)$, and
the two half-planes of support of $TP_a$ are separated by distance $a$. 
In higher dimensions we can expect some kind of zoo of entire solutions
starting with the conic, rotationally symmetric solution in \cite{AC}. 
For that reason, we will confine ourselves for most of this paper to the plane.

The
new development in the theory of free boundaries that allowed  us
to make progress in the two-dimensional case is the discovery
by Hauswirth, H\'elein and Pacard \cite{HHP} of another family of entire solutions
constructed via conformal mapping as follows.  Set 
\[
\varphi (\z) = \z + \sinh \z.
\]
When restricted to the strip $\mathcal{S}=\{\zeta \in \C: |\Im \zeta| < \pi/2 \}$, $\varphi$ is biholomorphic onto its image 
$$
\Omega_1 := \varphi (\mathcal{S}) = \{z = x_1 + i x_2 \in \C: |x_2 |< \pi/2 + \cosh (x_1) \}
$$ 
and the function given by
\[
 H(z) :=  \left\{ \begin{array}{ll} \Re (\cosh(\varphi^{-1}(z))) & z\in \Omega_1 \\ 0 & \text{otherwise}  \end{array} \right.
\]
solves \eqref{FBP}.   We obtain a family of solutions by taking dilates $H_a(z) = a H_1(z/a)$, $a>0$.
The positive phase region $\O_a = \{z\in \C: H_a(z)>0\} = a\O$ is the region between two catenaries
separated by distance $(2+\pi)a$ as shown in Figure \ref{hairpins}.   We call these {\em double hairpin} 
solutions because the catenaries resemble hairpin curves.  The 
family of double hairpin solutions is completed by taking rigid motions of the regions $\O_a$ and corresponding
functions $H_a$. 
\begin{figure}[h]
\centering \includegraphics[scale=0.6]{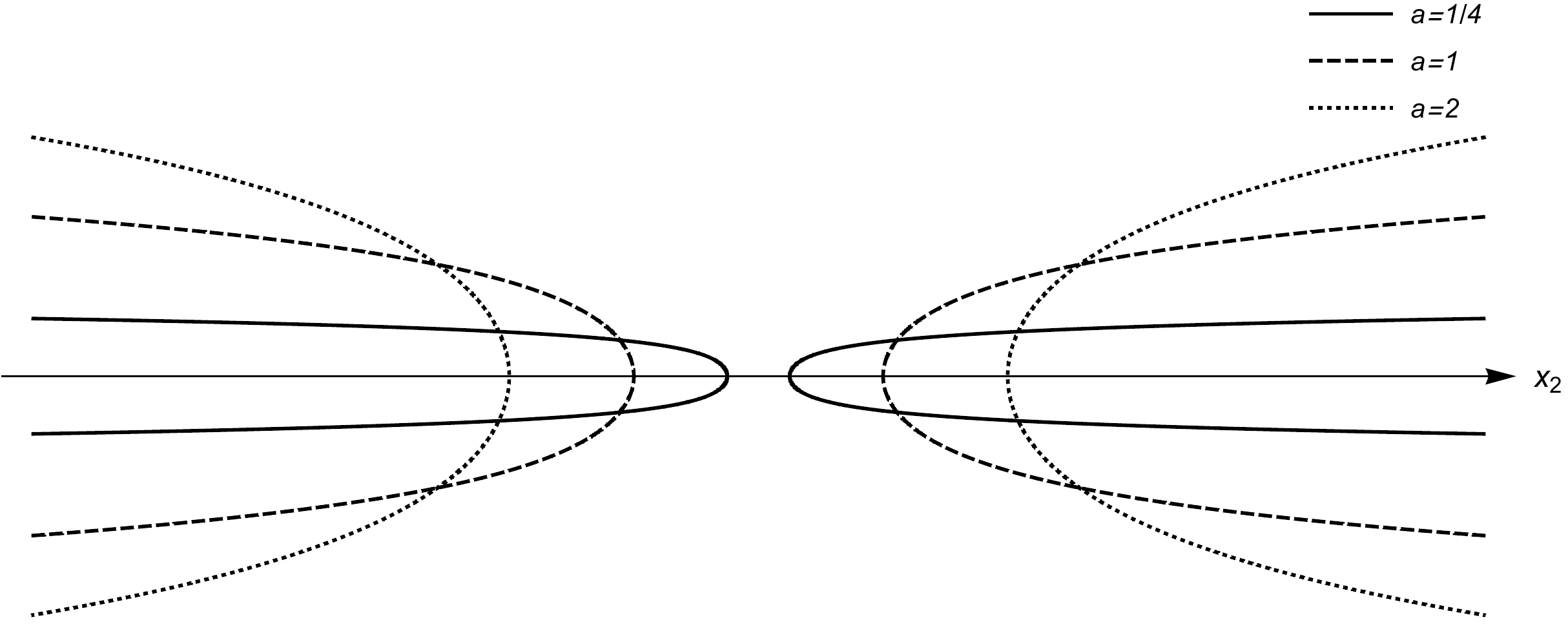} \caption{Mathematica plot of the free boundary of the double hairpin solution $H_a(z)$ for $a=1/4$, $a=1$ and  $a=2$. Note that $z=x_1 + i x_2$ and $x_2$ is the horizontal axis in the diagram.}
\label{hairpins}
\end{figure}\noindent

The double hairpin solutions, along with the one and two-plane solutions
are the only classical, entire, simply connected solutions in the plane.  Depending
on the definition of a solution, this result was proved by Khavinson, Lundberg and Teodorescu \cite{KhavLundTeo} and by  Traizet \cite{Traizet}.   For our work we require a different 
version of this characterization of entire solutions,  namely those
that arise as uniform limits of classical solutions (Theorem \ref{thm:mainglobal}).

The consequence in \cite{JK} of our characterization of entire solutions
is that every classical solution
in $\DD$ with {\em simply connected} positive phase is nearly isometric on
sufficiently small interior disks to one of these three entire solutions.   
Suppose that $u$ is a classical solution in $\DD$, $\DD^+(u)$ is simply connected, 
$z\in F(u)$, $|z| \le 1/2$, and $r>0$ is less than some absolute constant.  Then
there are never more than two strands of $F(u)$ in $B_r(z)$, and there are,
roughly speaking, three cases (see Figure \ref{FB_3cases}):

\begin{figure}[h]
\centering \scalebox{.7}{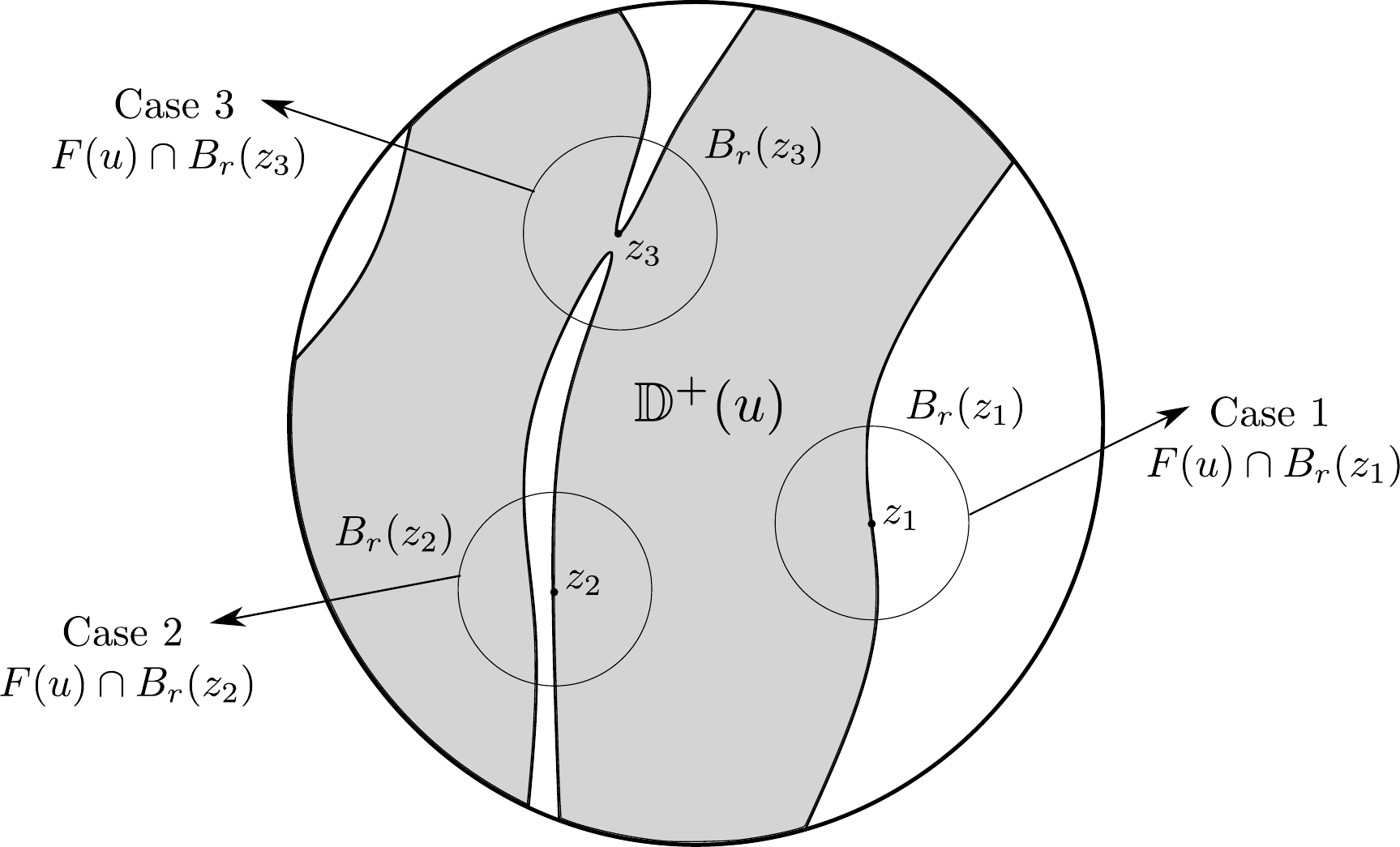} \caption{An illustration of the three cases for the free boundary $F(u)\cap B_r(z)$ of a solution $u$ having simply connected positive phase $\DD^+(u)$.}
\label{FB_3cases}
\end{figure}

\noi
{\em Case 1.}   $F(u)\cap B_r(z)$ 
has one strand (one connected component),
and $\DD^+(u) \cap B_r(z)$ is connected. 
Moreover, the curvature of that strand is bounded by an absolute constant,  
and the free boundary and positive
phase resemble those of a one-plane solution. 

\noi
{\em Case 2.}  $F(u)\cap B_r(z)$ 
has two strands, and $\DD^+(u)\cap B_r(z)$ 
has two components.   The curvature of each strand is bounded by an absolute
constant, and  the free boundary and positive phase resemble those of a two-plane solution
with zero phase in between the two strands.

\noi
{\em Case 3.} 
$F(u)\cap B_r(z)$ 
has two strands, and $B_r(z)\cap \DD^+(u)$ 
has only one component.  The free boundary and positive phase resemble
those of a double hairpin solution with positive phase in between the two strands.

In all three cases, the resemblance is expressed as an approximate
isometry with distances distorted by at most, say, one percent at all 
scales up to an absolute constant times unit scale.   
Precise statements are given in Lemma \ref{lem:flattypes} and Theorems \ref{thm:mainglobal} and \ref{thm:mainsimplyconn}.
In Case 3, we obtain more detailed estimates of the curvature of $F(u)$, 
but the resemblance of the curvature with that of the double hairpin is only
valid in $B_{\sqrt{a}}(z)$, not all the way to unit scale, as we will discuss next in connection with minimal surfaces. 


With Figures \ref{hairpins} and \ref{FB_3cases}  in mind, we can describe the analogy with minimal
surfaces.  We will do this in reverse order, starting with Case 3.  
Consider an embedded, minimal surface $M$ in the open
unit ball $\B\subset \R^3$ homeomorphic
to an annulus and such that $\de M$ consists of two disjoint, simple loops
in $\de \B$.  Suppose that there is a shortest geodesic loop in $M$ of 
length $a$.  The number $a$ is known as the {\em neck size}.
Colding and Minicozzi have proved that $M$ is very close
to a standard catenoid in the following sense. 

\begin{theorem}  (Colding and Minicozzi \cite[Theorem 1.1 and Remark 3.8]{CMAnnuli}) \label{thm:catenoid} 
Suppose that $M$ is  minimal, embedded annulus in the unit
ball $\B$ with $\de M \subset \de \B$ as above.   If the neck size $a$ of $M$ is less than some absolute constant,
and the shortest geodesic loop passes through $x$, $|x|\le 1/2$, then 
there is a diffeomorphism of $M\cap B_{\sqrt{a}}(x)$ to the standard catenoid 
with neck size $a$ that preserves distance and curvature to within one percent. 
\end{theorem}

A minimal, embedded annulus in the ball in 3-space 
with boundary on the 2-sphere is analogous to a free boundary in the disk 
consisting of two strands as in Case 3 above, with ends on the boundary circle
and positive phase in between.  In our theorem, we have separated
the issue of approximation to first order (isometry) and to second order (curvature).
Our free boundary solution is approximately isometric to a pair
of catenaries up to unit scale. 

The minimal surface theorem does not address the
issue of approximate isometry up to unit scale because it is aimed
at rigidity of curvature rather than distance.   It only matches the approximate catenoid with the standard 
one up to scale $\sqrt{a}$.  Our free boundary second derivative (or curvature) bounds have a similar limitation.  
They are valid up to unit scale, but only yield a good approximation of curvature (within one percent, say)
up to scale $\sqrt{a}$, analogous to the minimal surface case. 
They are more delicate than our isometric
approximation because they require us to match our free boundary strands
to the catenaries with greater precision.   One
can regard the neck size of the annulus as analogous to the distance between 
the two strands.  In our proof we use a different normalization, namely 
the value $a$ of the solution $u$ at the (unique)
saddle point between the two strands.  This parameter $a$
identifies which double hairpin solution to choose, namely
the one with the same value at its saddle point, $a = H_a(0)$.

In analogy with Case 2, we consider a minimal surface consisting
of two nearby minimal disks, as follows.  
\begin{theorem}  \label{thm:twoflatplanes} (Colding and Minicozzi) Suppose that $M = D_1 \cup D_2\subset \B$ is an embedded minimal surface with two connected components $D_j$, each homeomorphic to a disk with $\de D_j$ simple disjoint loops in $\de \B$.  
There is an absolute constant $c>0$ such that 
if both components $D_j$ intersect the ball $B_c$ of radius $c$ centered at $0$,
then after rotation, the surfaces $D_j \cap B_{c}$ are Lipschitz graphs:
\[
D_j \cap B_{c} = \{(x',y)\in \R^2 \times \R: y = g_j(x')\} \cap B_{c}
\]
with, for example, $|\nabla g_j| \le 1/100$.  
\end{theorem}
\begin{proof}   It is not hard to show that there is a stable minimal disk in between the two embedded minimal
disks.   The theorem then follows from the ``one-sided curvature bound" Theorem 0.2 of  \cite{CM4}.
\end{proof}

In Theorem \ref{thm:twoflatplanes}, one obtains
surfaces resembling two parallel planes.   By analogy, in Case 2, the free boundary is approximated by 
two straight lines.   (A well known consequence is higher regularity bounds like $|\nabla^k g_j|  \le C_k$, but 
the first-order gradient bound is the most important because it is scale-invariant and
must be proved first.)

The minimal surface theory setting analogous to Case 1 is a single, 
embedded, minimal
disk in the ball $\B$ whose boundary is a simple loop in $\de \B$.
But in this situation there is a new family of  minimal
surfaces not present in the free boundary theory, namely helicoids.    
The other major result of Colding and Minicozzi in \cite{CM4}  says that all smooth,
embedded, minimal disks are approximated, in an appropriate sense, either by a plane 
or by a helicoid.   See also the work of Meeks and Rosenberg \cite{MeeksRosenbergHeli} characterizing 
complete, embedded minimal disks.   The absence of a family like the helicoids 
from  the two-dimensional  free boundary theory indicates that the proofs 
should be significantly easier than in the minimal surface case.

 There is also another intermediate case arising in parallel in
 both free boundary and minimal surface theory.  An important
 step in the theorems of Colding and Minicozzi is what they
 call an effective removable singularities theorem for minimal
 surfaces.  Removable singularities theorems (\cite{ChoiSchoen}, \cite{White}) say that if
 an elliptic equation such as  the minimal surface equation (mean curvature zero) is satisfied at all but one point, 
 then  it is satisfied everywhere.  In the more quantitative, so-called effective, form, the minimal
 surface is undefined inside a small ball of radius $\d>0$. 
 
 \begin{theorem} (Colding and Minicozzi \cite{CMAnnuli} and \cite{CM3}) \label{thm:removable-min} 
 Suppose that $M$ is an embedded, minimal annulus in $\d < |x| < 1$, one
 component of $\de M$ is a simple loop in $|x| = \d$, and the other
 is a simple loop in $|x| = 1$. 
  Then there are absolute constants $\d_0$ and
 $C$ such that if $\d \le \d_0$, then, after rotation,
 \[
M \cap  \{ C\d < |x| < 1/C\} 
=   \{(x',y)\in \R^2 \times \R: y = g(x')\} \cap  \{ C\d < |x| < 1/C\} 
   \]
   with $|\nabla g| \le 1/100$.  
 \end{theorem}
 \noi
We emphasize the uniformity of the gradient bound up to the scale $\d$ as $\d \to 0$. 

An  analogous free boundary theorem is as follows.
\begin{theorem} \label{thm:removable}  Suppose that $u$ is a
classical solution to the one-phase free boundary problem in the annulus
$\{z\in \C: \d < |z| <1\}$ and $F(u)$ consists of two disjoint connected curves
each with one end on $|z| =1$ and the other on $|z| = \d$. 
For any $\e>0$ there exist $\d_0>0$ and $C>0$ such that 
if $T$ denotes the connected component of $\{z\in \C: \d < |z| <1/C,\ u(z)>0\}$ 
whose closure intersects the circle $|z| = \d$, then for
all $\d$, $0< \d < \d_0$, there is a rotation such that
 \[
T \cap  \{ C\d < |z| < 1/C\} 
=   \{z = x+iy\in \C: y > g(x)\} \cap  \{ C\d < |z| < 1/C\} 
   \]
with $|\nabla g| \le \e$.  
\end{theorem}

Theorem \ref{thm:removable} is new to this paper.   We proved
a special case in \cite{JK} with an extra hypothesis of flatness as opposed 
to purely topological assumptions on the free boundary.   That theorem is 
the effective removable theorem analogue of a ``flat implies Lipschitz"
type theorem, or, put another way ``flat with a small hole implies Lipschitz". 

The paper is organized as follows.  In Section 2 we state the main
definitions and results.  In Section 3 we  sketch the proof of 
Theorem \ref{thm:mainglobal}, the global classification of limits of classical
solutions in dimension 2. In Section 4 we prove the quantitative removable
singularities theorem, Theorem \ref{thm:removable}.   In Section 5, we
describe the double hairpin solution of Hauswirth {\em et al.} in terms of a
conformal mapping based on the solution itself.   In Section 6, we use
this description to sketch the proof of our main classification theorem
from \cite{JK} of  classical solutions with simply connected positive phase, Theorem \ref{thm:mainsimplyconn}.
In Section 7, we discuss multiply connected solutions in the plane and Traizet's
correspondence with minimal surfaces, and we find a simple formula, equation \eqref{eqn:ScherkBubble}, 
for the multiply connected free boundary associated to the Scherk simply periodic surface.  In Section 8, we discuss open problems in dimensions 2 and 3.  In particular, even in dimension 2, there is unfinished business 
related to multiply connected classical solutions, as well as solutions that arise from
Caffarelli's Perron construction in \cite{CafIII}. 

We thank Toby Colding for explaining to us his proof with W. Minicozzi of Theorem \ref{thm:twoflatplanes},
and we thank Christos Mantoulidis for discussions about elliptic semilinear equations. 

\section{Main classifications results}

Here we state the main results of \cite{JK}.    In subsequent sections
we will sketch the proofs in sufficient detail to be able to make use of those
methods in the proof of the new result, Theorem \ref{thm:removable}.

Following \cite{CafIII} and \cite{CafSalsa}, we define viscosity solutions. 
\begin{definition}  \label{def:visc.super} A viscosity supersolution of \eqref{FBP} is a non-negative
continuous function $w$ in $\O\subset \R^n$ such that $\D w \le 0$ in $\O^+(w)$ and 
for every $x_0\in F(w)$ with a tangent ball $B$ from the positive side
($x_0\in \de B$ and $B\subset \O^+(w)$), there is $\a \le 1$ such that
\[
u(x) = \a  \langle x-x_0, \nu\rangle^+ + o(|x-x_0|) 
\]
as $x\to x_0$ non-tangentially in $B$, with $\nu$ the inner normal to $\de B$ at $x_0$. 
\end{definition}
\noi
In other words, $u$ is subharmonic in its positive phase and has slope less than or equal
to one on the boundary at points with interior tangent balls. 

\begin{definition} \label{def:visc.sub}   A viscosity subsolution of \eqref{FBP} is a non-negative
continuous function $w$ in $\O\subset \R^n$ such that $\D w \ge 0$ in $\O^+(w)$ and 
for every $x_0\in F(w)$ with a tangent ball $B$ in the zero set
($x_0\in \de B$ and $B\subset \{w=0\}$), there is $\a \ge 1$ such that
\[
u(x) = \a  \langle x-x_0, \nu\rangle^+ + o(|x-x_0|) 
\]
as $x\to x_0$ non-tangentially in $B^c$, with $\nu$ the outer normal to $\de B$ at $x_0$. 
\end{definition}

A {\em viscosity solution} in $\O$ is a function that is both a supersolution and a subsolution in
the sense above.   Note that this definition allows for generalized solutions of the form
\[
\a x_1^+ + \b x_1^-, \quad 0 \le \b \le \a  \le 1.
\]
The only classical solution among these is the half-plane solution, the case $\b=0$, $\a  = 1$. 

We also make use of another notion of weak solution introduced by G. Weiss \cite{Weiss1}.

\begin{definition}  \label{def:varsolution}  A function $u\in H^1_{\text{loc}}(\O)$ is a variational solution of \eqref{FBP} if
$u\in C(\O)$, $u\in C^2(\O^+(u))$, and $u$ satisfies the Euler-Lagrange identity in integrated form:
\[
\int_\O (|\nabla u|^2 + 1_{\{u>0\}} ) \mbox{\rm div} \, \psi - 2 (\nabla u) D\psi (\nabla u)^T \, dx = 0
\]
for every test vector field $\psi \in C_0^\infty (\O, \R^n)$. 
\end{definition}

\begin{theorem} \label{thm:mainglobal} (\cite[Theorem 1.1]{JK})  Let $u_k$ be a sequence of classical solutions to \eqref{FBP} in disks $B_{R_k}\subset \R^2$
centered at $0$ with radii $R_k\to \infty$ as $k\to \infty$, then there is a subsequence that converges uniformly on compact subsets
to a function $U$ defined on all of $\R^2$.  Suppose further that $0\in F(u_k)$ and each connected curve of $F(u_k)$ has both ends on $\de B_{R_k}$. Then up to rigid motion $U$ is either the half plane solution $P$, 
a two plane solution $TP_a$, for some $a>0$, a double hairpin $H_a$ for some $a>0$, or
a wedge $W_1(x) = |x_1|$.
\end{theorem}

\noi
Note that the property that the connected curves of $F(u_k)$ end at $\de B_{R_k}$ says that
each connected component of the positive phase $B_{R_k}^+(u_k)$ is simply connected.
The wedge $W_1(x)$ is the limit as $a\to 0$ of both the two plane solution $TP_a$ and the hairpin
$H_a$.

\begin{theorem}  \label{thm:mainsimplyconn}  (\cite[Theorem 1.2]{JK})  If $u$ 
is a classical solution to \eqref{FBP} in
the unit disk, $\DD$, and the positive phase $\DD^+(u)$ is connected 
and simply connected, 
then there exist absolute constants $C_0$ and $c>0$ such
that either

a)  at most one component of the zero phase $\DD^0(u) := \{z\in \DD: u(z) = 0\}$ 
intersects $B_c$, the disk of radius $c$ centered at $0$, $F(u) \cap B_c$ consists
of at most two simple curves, and the curvature of $F(u)$ at every point of
$B_c$ is bounded by $C_0$, or

b) exactly two components of $\DD^0(u)$ intersect $B_c$, and the picture
is almost isometric to a double hairpin in the following sense.

There is a unique
saddle point $z_0$ of $u$ in $B_{2c}$, and, setting
$a = u(z_0)$, the two components of $F(u)$ that intersect
$B_c$ are at
a distance apart comparable to $a$.  Furthermore,
there is a conformal 
mapping 
\[
\Psi: B_{4c}\cap \O_a \to  D^+(u),  \quad \Psi(0) = z_0,
\]
that is close to an isometry: for some $\th \in \R$,
\[
|\Psi'(z) - e^{i\th}| \le 1/100, \quad \mbox{for all} \ 
z \in  B_{4c} \cap \O_a.
\]
The mapping $\Psi$ extends smoothly to all
$z\in B_{4c}\cap \de \O_a$,
and  for all such $z$,
\[
\Psi(z) \in F(u),  \quad |\Psi'(z)| = 1.
\]
Finally, we have curvature bounds: 
\[
|\Psi''(z)| \le 1/100; \ z\in B_{4c}\cap \O_a;  \quad
|\k(\Psi(z)) - \k_a(z)| \le 1/100, \ z\in \de B_{4c} \cap \de \O_a,
\]
with $\k$ the curvature of $F(u)$ and $\k_a$ the curvature
of $\de \O_a$. 
\end{theorem}

\section{Global limits of classical free boundary solutions}

In this section we outline the proof from \cite{JK} of the characterization of 
global limits of classical solutions Theorem \ref{thm:mainglobal}.   
The reader should refer to the original for further details.  This presentation
emphasizes the elements we will need to revisit in order to prove
the removable singularities result of the next section.

\begin{lemma}[Lipschitz bound] \label{lem:Lipschitz1}
There is an absolute constant $C$ such that
if $u$ is a classical solution in the disk $B_R$ and $0\in F(u)$, then
\[
\max_{|z|\le R/2} |\nabla u(z)| \le C.
\]
\end{lemma}
This lemma is proved in \cite{AC} using the Hopf lemma.
In particular, this family of classical solutions $u_k$  in $B_{R_k}$ 
is equicontinuous and hence precompact in the uniform norm on compact subsets. 
Therefore, there is a uniformly convergent subsequence.  Thus we obtain
the existence of a limit function $U$ as in Theorem \ref{thm:mainglobal}.

Our next lemma shows $u_k$ is non-degenerate, uniformly in $k$.

\begin{lemma}[Non-degeneracy] \label{lem:nondegeneracy1} 
If $u$ is a classical solution in $B_R$ 
and, in addition, the positive phase of $u$ in $B_R$ 
is simply connected, then there is an absolute constant $c>0$ such that for any $z\in F(u)$ and any radius $r>0$ such that $B_r(z)\subseteq B_R$
\[
\sup_{B_r(z)}u \ge c r.
\]
\end{lemma}
\noi 
In other words, $u$ grows linearly away from the free boundary.  This
non-degeneracy statement is Proposition 3.4 of \cite{JK}.  (Simple connectivity 
implies that each connected curve of 
$F(u)$ ends on $\de B_R(0)$.)

The next proposition is well known (see \cite[Lemma 1.21]{CafSalsa}, \cite[Proposition 4.2]{JK}) and says that uniform limits of Lipschitz, non-degenerate variational/viscosity solutions inherit those same properties.  In particular, $U$ has these properties. 
\begin{prop}\label{prop_limitofsolns} Let $\{u_k\} \in H^1_{\text{loc}}(\O) \cap C(\O)$ be a sequence of variational (respectively, viscosity) solutions of
\eqref{FBP} which satisfies
\begin{itemize}
\item (Uniform Lipschitz continuity) There exists a constant C, such that $\|\nabla u_k\|_{L^{\infty}(\O)} \leq C$;
\item (Uniform non-degeneracy) There exists a constant c, such that $\sup_{B_r(x)} u_k \geq cr$ for every $B_r(x) \subseteq
\O$, centered at a free boundary point $x\in F(u_k)$.
\end{itemize}
Then any limit $u \in H^1_{\text{loc}}(\O) \cap C(\O)$ of a uniformly
convergent on compacts subsequence $u_k \to u$ satisfies
\begin{enumerate}[(a)]
\item $\overline{\{u_k>0\}} \rightarrow \overline{\{u>0\}}$ and $F(u_k) \rightarrow F(u)$ locally in the Hausdorff distance;
\item $1_{\{u_k>0\}} \rightarrow 1_{\{u>0\}}$ in $L^1_{\text{loc}}(\O)$;
\item $\nabla u_k \rightarrow \nabla u$ a.e. in $\O$.
\end{enumerate}
Moreover, $u$ is a Lipschitz continuous, non-degenerate variational (resp. viscosity)
solution of \eqref{FBP}.
\end{prop}

Other auxiliary results we will need are the characterization of blow-ups
and blow-downs of our weak solutions. 
\begin{prop}[Characterization of blow-ups]\label{prop_blowups} Let $u$ be a variational solution of \eqref{FBP} in $\O\subseteq \R^2$, which is
Lipschitz-continuous and non-degenerate. Assume $0 \in F(u)$ and let
$v$ be any limit of a uniformly convergent on compacts subsequence
of
\begin{equation*}
    v_j(x) =  r_{j}^{-1}u(r_jx)
\end{equation*}
as $r_j \to 0$. Then, up to rotation, $v$ is either $P$ or $W_s$ for some $s$, $0 < s <\infty$, in an appropriately chosen Euclidean coordinate system. 

If, in addition, $u$ is a viscosity solution, then so is $v$ and, up to rotation, either $v = P$, or $v = W_s$ for some $0<s\leq 1$.
\end{prop}
\begin{proof}
As a consequence of Proposition \ref{prop_limitofsolns}, $v$ is a Lipschitz continuous, non-degenerate variational solution of \eqref{FBP}. Furthermore, the Weiss Monotonicity Formula \cite{Weiss1} implies that the blow-up $v$ is homogeneous of degree one. Thus, after possibly rotating the
coordinate axes
\begin{equation*}
    v(x) = s_1 x_1^+ + s_2 x_1^-,
\end{equation*}
where $s_1 \geq s_2 \geq 0$. We have the following two cases.

\paragraph{\textbf{Case 1} ($s_2 = 0$)} By non-degeneracy we must have $s_1 > 0$ and the variational solution condition \eqref{def:varsolution} implies that
\[
0 = L[v](\phi) = \int_{\{x_1>0\}} \left((s_1^2 + 1) \text{div } \phi - 2s_1^2 \de_{x_1} \phi_1 \right)dx =  
(s_1^2-1) \int_{-\infty}^\infty  \phi_1(0,x_2) \, dx_2
\]
for any $\phi = (\phi_1, \phi_2)\in C^1_c(\R^2;\R^2).$ Thus, $s_1 = 1$ and $v =P$.

\paragraph{\textbf{Case 2} ($s_2 >0$)} In this case the variational solution condition says
\begin{align*}
    0  = L[v](\phi) &= \int_{\{x_1>0\}} \left((s_1^2 + 1) \text{div} \phi - 2s_1^2 \de_{x_1} \phi_1 \right)dx + \int_{\{x_1<0\}} \left((s_2^2 + 1) \text{div} \phi - 2s_2^2 \de_{x_1} \phi_1 \right)dx \\
    & = (s_1^2 -s_2^2)\int_{-\infty}^\infty  \phi_1(0,x_2) \, dx_2
\end{align*}
for any $\phi = (\phi_1, \phi_2)\in C^1_c(\R^2;\R^2).$ Thus, $s_1 = s_2 = s>0$ and $v =W_s$. If, in addition, $u$ is a viscosity solution in $\O$, then by Proposition \ref{prop_limitofsolns}, $v=W_s$ is a non-degenerate viscosity solution, as well. Since every point of $ F(v) = \{x_1=0\}$ has a tangent disk from the positive phase of $v$ only, the viscosity supersolution condition implies $s\leq 1$.
\end{proof}

The following characterization of blow-down limits has a closely analogous proof.

\begin{prop}[Characterization of blow-downs]\label{prop_blowdowns} Let $u$ be a variational solution of \eqref{FBP} in all of $\R^2$, which is
Lipschitz-continuous and non-degenerate. Assume $0 \in F(u)$ and let
$w$ be any limit of a uniformly convergent on compacts subsequence
of
\begin{equation*}
   w_j(x) = R_{j}^{-1}u(R_jx)
\end{equation*}
as $R_j \to \infty$. Then, up to rotation, $w$ is either $P$ or $W_s$ for some $s>0$ in an appropriately chosen
Euclidean coordinate system. 

If, in addition, $u$ is a viscosity solution, then so is $w$ and, up to rotation, either $w= P$, or $w = W_s$ for some $0<s\leq 1$.
\end{prop}

From Proposition \ref{prop_blowdowns}, any blow down limit of $U$,
\[
V(x) = \lim_{j\to \infty} (1/R_j) U(R_j x),
\]
is, after rigid motion, either a one-plane solution $P(x)= x_1^+$ or a wedge solution
$W_s(x) = s|x_1|$ for some $0 < s \le 1$.   In all cases,  the convergence
given by Proposition \ref{prop_limitofsolns} implies that
the free boundary $F(U)$ is asymptotically flat in the sense that
there exist $R_j\to \infty$ and $\d_j \to 0$ such that
\[
F(U)\cap B_{R_j}  \subset \{|x_1| \le \d_j R_j\}
\]
The next lemma gives a preliminary characterization of flat free boundaries
which we will use at large and small scales.


\begin{lemma}  \label{lem:flattypes}  
For any $\e>0$ there is $\d>0$ such the following holds.  
Suppose that $u$ is a  classical solution to the free boundary
problem \eqref{FBP}  in $B_3 \subset \R^2$ such that 
every component curve of $F(u)$ has both ends on $\de B_3$.
Suppose further that the Hausdorff distance from $F(u)$ to the vertical segment
$\{(0,x_2): |x_2| < 3\}$ is less than $\d$.  Then either

a)  $B_1^+(u)$ and $B_1^0(u)$ are both connected, and 
\[
B_1\cap F(u) = \{ x\in B_1: x_1= g(x_2)\}
\]
for some $g:[-1,1]\to \R$ satisfying $|\nabla^j g | \le \e$, $j=0, \, 1, \, 2, \, 3$;

b)  $B_1^+(u)$ has two components, $B_1^0(u)$ has one component, and
\[
B_1^0(u) = \{ x\in B_1: g_1(x_2) \le x_1 \le g_2(x_2)\}
\]
for some $g_i:[-1,1] \to \R$, satisfying $g_1(s) < g_2(s)$ and
$|\nabla^j g_i | \le \e$, $j=0, \, 1, \, 2, \, 3$;

c)  or $B_2^+(u)$ is connected, and $B_2^0(u)$ has two components,
one bounded by a simple curve both of whose ends belong to
$\a_-(2)$ and the other 
bounded by a simple curve both of whose ends belong to $\a_+(2)$,
where $\a_\pm(r)$ are the top and bottom circular arcs of 
 $\de B_r \cap \{|x_1| \le \d\}$.
\end{lemma}

\begin{proof}  Case 1.  $F(u)\cap B_{4/3}$ is a simple curve connecting
$\a_-(4/3)$ to $\a_+(4/3)$.  Then the one-phase ``flat implies smooth"
free boundary regularity theorems of \cite{AC} imply that we are in case (a).  

\noi
Case 2.  $F(u)\cap  B_{4/3}$ 
consists of two simple curves connecting 
$\a_-(4/3)$ to $\a_+(4/3)$ with a zero phase in between.
Then the positive phase has two components.  Considering each of
them separately, the same estimates of \cite{AC} as in Case 1 show
that we are in case (b).

\noi
Case 3.  $F(u) \cap B_{4/3}$ consists of 
two simple curves $C_1$ and $C_2$ with positive phase 
in between them.   To rule out this possibility we first prove the following.
\begin{lemma}[Flux Balance] \label{lem:fluxbalance}  Let $\O$ be
a bounded, piecewise $C^1$ domain, and let $u$ be
a harmonic function in $\O$ that is $C^1$ up to the boundary
and such that $|\nabla u| \le L$.   Denote by $\nu$ the outer
unit normal to $\de \O$, and $F \subset \{x\in \de \O: \nu \cdot \nabla u = -1\}$.
Then
\[
\HH^1(F) \le L \, \HH^1((\de \O) \setminus F) \quad (\text{with $\HH^1$ equal to 1-dimensional Hausdorff measure}).
\]
\end{lemma}
\begin{proof} By the divergence theorem,
\[
0 = \int_{\O}  \mbox{\rm div}(\nabla u) \, dx = \int_{\de\O} \nu \cdot \nabla u \, d\HH^1
= -\HH^1(F) + \int_{(\de\O)\setminus F} \nu \cdot \nabla u \, d\HH^1
\]
The estimate follows from the bound $|\nu \cdot \nabla u| \le L$. 
\end{proof}

Returning to Case 3, let $\O$ be the subset of $B_{4/3}$ between
the two curves of $F(u)\cap  B_{4/3}$.   
The remainder of the boundary
is a subset of $\a_\pm(4/3)$.    By Lemma \ref{lem:Lipschitz1},
the function $u$ of Lemma \ref{lem:flattypes}
satisfies $|\nabla u | \le L$ on $B_{5/2}$ for some absolute constant $L$. 
By Lemma \ref{lem:fluxbalance}, 
\[
4 \le \HH^1 (C_1 \cup C_2) \le L \, \HH^1(\a_+(4/3) \cup \a_-(4/3)) \le 2 L\, \d.
\]
This is impossible if $\d$ is sufficiently small ($\d < 2/L$).

\begin{figure}[h]
\centering 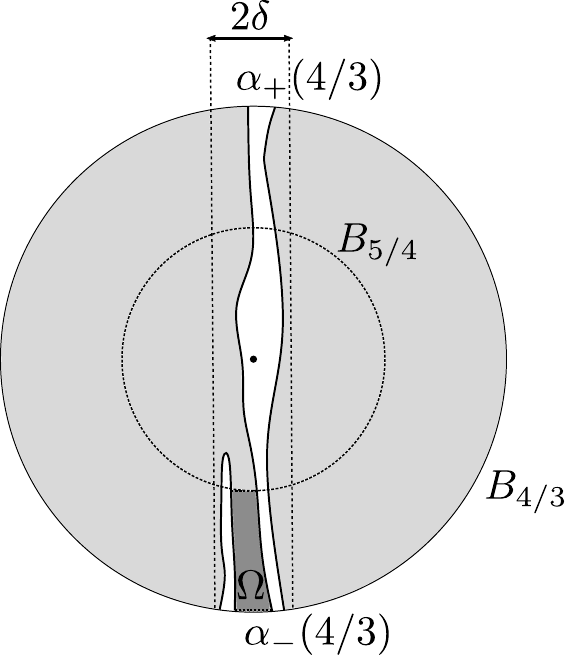 \caption{Illustrating $\O$ in Case 4 of the proof of Lemma \ref{lem:flattypes}, whose existence is ruled out.}
\label{Fig:lem:flattypes}
\end{figure}

\noi
Case 4.  $F(u) \cap B_{4/3}$ contains two simple curves that 
join $\a_+(4/3)$ to $\a_-(4/3)$.  In this case, we will show that (b) is valid
by showing although there may be additional curves of $F(u)$ in $B_{4/3}$ there
are none that penetrate closer to the origin than distance $5/4$.  
Suppose, to the contrary, there is another curve
of $F(u) \cap B_{4/3}$ that reaches $B_{5/4}$.   Then by hypothesis,
it must connect $\de B_{5/4}$ to one of $\a_\pm(4/3)$.  It follows
that there is a region $\O$ of positive phase in $(B_{4/3}\setminus B_{5/4})\cap \{|x_1| < \d\}$ of diameter at least $4/3 - 5/4 = 1/12$
bounded by two curves of $F(u)$ and two vertical segments of length
at most $\d$ (see Figure \ref{Fig:lem:flattypes}).  This contradicts Lemma \ref{lem:fluxbalance} if $\d < 1/12 L$. 
Therefore, there are exactly two curves of  $F(u)\cap B_{5/4}$
each of which has one end in $\a_+(5/4)$ and the other in $\a_-(5/4)$.
Thus, the method of Case 2 applied to $B_{5/4}$ instead of $B_{4/3}$
implies we are once again that we have property (b).

\noi
Case 5.  The argument in Case 4 rules out the possibility of more
than $2$ curves connecting $\a_+(4/3)$ to $\a_-(4/3)$, so it only remains to
consider the case in which there is no curve of $F(u)\cap B_{4/3}$ connecting $\a_+(4/3)$ to $\a_-(4/3)$.  In that case, we will show that (c) is satisfied.
Since $F(u)$ is within $\d$ Hausdorff distance of the segment 
$\{(0,x_2): |x_2| < 3\}$,  and every component of $F(u)$ has
its ends on $\de B_3$,  there are two simple curves $C_\pm \subset F(u)$
such that $C_\pm \cap B_{4/3}$ is nonempty, $C_+$  has both ends on 
$\a_+(3)$, and $C_-$ has both ends on $\a_-(3)$.   Let $p_\pm = (\pm 2, 0)$.
By a similar flux balance argument to Case 4, 
$F(u) \cap B_2 = (C_+\cup C_-) \cap B_2$.  Moreover,
$F(u) \cap B_{1/10}(p_\pm)= C_\pm \cap B_{1/10}(p_\pm)$  is the union of two graphs with $u=0$ in between,
the same as case (b) rescaled by $1/10$ and translated to be centered at $p_\pm$.
In particular, $C_+ \cap B_2$ is a simple curve with both ends on $\a_+(2)$ and
similarly for $C_- \cap B_2$, confirming that we are in case (c).  This completes
the proof of Lemma \ref{lem:flattypes}.
\end{proof}

The next proposition summarizes Propositions 7.1, 7.2 and 7.3 of \cite{JK}.
\begin{prop} \label{prop:Voutcomes}  If $V$ is a blow down limit of $U$, then
there are three possibilities, after rigid motion

a)  $V$ is a one plane solution, and $U = x_1^+$;

b)  $V = W_s$ for some $s$, $0 < s  \le 1$, and $U$ is a two plane solution
$TP_a$ for some $a>0$, or $U$ is a wedge solution $W_1$;

c)  $V = W_s$ for some $s$, $0 < s  \le 1$, and 
$U$ is a double hairpin solution.
\end{prop}

\begin{proof}  For part (a) 
Now we consider each possible outcome for $V$. 
If, after rigid motion, $V$ is a one-plane solution $x_1^+$ , then
by Proposition 7.1 \cite{JK}, $U  = x_1^+$. Briefly,
for any $\e>0$ there is $R$ such that for all $r>R$, 
for sufficiently large $R$, $\{U>0\} \cap B_r \subset \{x_1 \ge - \d r\}$ 
and the fact that $U = x_1^+$ follows from the regularity theory of Alt 
and Caffarelli as in Lemma \ref{lem:flattypes} (a).

If, after rigid motion, $V = W_s$ for some $s$, $0 < s \le 1$,
then by Proposition 7.2 of \cite{JK}, either $U$ is a two-plane solution $TP_a$ for some $a> 0$, or $U$ is a wedge solution $W_1$, or $|\nabla U| < 1$ for all points of $\{U>0\}$.  We  
summarize the argument. A majorization of $U$ using $W_1$
\cite[Lemma 6.1]{JK} shows that $|\nabla U| \le 1$
on all of $\{U>0\}$.  If, in addition, there exists a point at which $|\nabla U| = 1$, 
then since $|\nabla U|^2$ is subharmonic, we have  $|\nabla U| \equiv 1$ on a component of $\{U>0\}$.   It follows that $U$ is linear on this component.  Hence the complementary positive component is contained in a half space.  This in turn implies
the other component of $\{U>0\}$ is a half space and $U$ must be a two-plane
solution $TP_a$ for some $a> 0$, or a wedge $W_1$.  

The remaining case is that $V = W_s$ but $|\nabla U| < 1$ at some point of every connected component of
$\{U>0\}$. In that case Proposition 7.3 of \cite{JK} shows that $U$ is a double hairpin solution. 
We sketch the proof here.

\begin{lemma}  \label{lem:almostsmooth} If $V= W_s$ for some $s$, $0< s \le 1$,
and $|\nabla U| < 1$ at some point of each component of its positive phase, then $F(U)$ consists
of finitely many disjoint curves that are smooth at all but one point.  Moreover,
near those smooth points, $U>0$ on one side and $U=0$ on the other, and
$F(U)$ has strictly positive curvature, i.~e., the set $U=0$ is locally convex.
\end{lemma}
\begin{proof}  Since the blow-up of $U$ is $W_s$, for any $\d>0$, there is a large $R<\infty$ 
such that after rotation,  $\tilde u_k(x) = Ru_k(x/R)$ satisfies the hypothesis of Lemma \ref{lem:flattypes}
for all sufficiently large $k$.   In cases, (a) and (b), the Lipschitz norm of the graph of the boundary, 
which is invariant under scaling, is bounded by $\e$.   It follows from the standard Alt-Caffarelli
regularity theory that $F(u_k)$ and $F(U)$ are smooth at all points of $B_R(0)$. Thus
the only case in which $F(U)$ can be non-smooth is the situation of part (c) of the 
lemma, i.~e., for sufficiently large $k$, the zero set of $u_k$ in $B_R$ has two connected 
(noncompact) components $B_R^0(u_k) =  \CC_1(k) \cup \CC_2(k)$.  

Next, suppose that there is a point $p\in F(U)$ at which the positive phase is not smooth 
from (at least) one side.   Then the same lemma and reasoning above 
shows that the approximations $u_k$ at small scale near $p$ 
cannot be of types (a) or (b).  Type (c) implies in particular 
that for all $r>0$ sufficiently small and $k$ sufficiently large
$B_r(p)\cap \{u_k=0\}$ has two exactly components.  
Moreover,  the flux balance implies that these two components must
equal $\CC_1(k) \cap B_r(p)$ and  $\CC_2(k) \cap B_r(p)$,
with $\CC_1(k)$ and $\CC_2(k)$ the components at the larger scale $R$. 
Now suppose there are two points $p_i$, $i = 1, \, 2$, at which $F(U)$ is non-smooth.
We derive a contradiction as follows.  Set $S_i= S_i(k)$ equal to the shortest segment 
 joining $\CC_1(k) \cap B_r(p_i)$ to $\CC_2(k)\cap B_r(p_i)$. The segments $S_1$,
 $S_2$, and $F(u_k)$ enclose a component of $B_R^+(u_k)$
 of diameter at least $|p_1 - p_2|/2$.  On the other hand the lengths
 of $S_i$ are at most $2r$.   So this violates the flux balance lemma. 
 
 We have now shown that all but one point of $F(U)$ consists of smooth curves.
We now wish to show that  at each of these points $F(U)$ consists of one 
smooth curve with positive phase on one side and zero phase on the other.
What we have to rule out is the possibility of two smooth curves the become
tangent at some  points with positive phase in both directions normal to the curves 
at the tangent points and zero phase in between.  We will rule this out using curvature of 
the free boundary.

Recall that $|\nabla U| < 1$ at some point of each component of $\{U>0\}$.  Suppose
that $p$ is a smooth boundary point of this component.   Since $|\nabla U|^2\le 1$ is subharmonic
and achieves the value $1$ at $p$, the Hopf lemma implies that the outer
normal derivative $\de_\nu |\nabla U|^2 >0$ at $p$.  It follows that the curvature
of $F(U)$ is strictly positive at these points.    If there were another positive component
touching at $p$ from the other side, then it would also be bounded by a curve with positive curvature, 
which is impossible.   In other words,
$U$ must be zero on the other side of $F(U)$ from $\{U>0\}$ at $p$.
This concludes the proof of Lemma \ref{lem:almostsmooth} 
\end{proof}

Finally, we conclude the proof of Theorem \ref{thm:mainglobal}.    
Consider any connected component $\CC$ of $U=0$.  We
have just shown that $\CC$ is locally convex except possibly at one
point $p$, and bounded by two smooth curves of strictly positive curvature,
except possibly at the point $p$.   Lemma 7.6 of \cite{JK} says that for
sufficiently small $r>0$, $B_r(p)\cap \CC$ contains a nontrivial sector.
It follows from Proposition \ref{prop_blowups} that the blow up of $U$ 
at $p$ is the half-plane solution.  Thus $F(U)$ is regular at every point,
and $U$ is a classical solution.   Last of all, the only classical, entire solutions
with simply connected positive phase besides the half-plane are double  hairpin solutions
(see \cite{Traizet}).

\end{proof}


\section{A quantitative removable singularities theorem for free boundaries}

In this section we will give a proof of the effective removable singularity Theorem \ref{thm:removable}. The theorem concerns classical solutions $u$ of \eqref{FBP} in the plane defined on annuli $B_{R_2}\setminus B_{R_1}$, whose free boundary satisfies the following assumption:
\[
\tag{A} 
\begin{array}{c}
F(u) \text{ consists of two connected components, each of which has} \\ \text{one end in } \de B_{R_1} \text{ and one end in } \de B_{R_2}.
\end{array}
\]

As in Lemma \ref{lem:Lipschitz1} such solutions $u$ satisfy a universal Lipschitz bound.

\begin{lemma}[Lipschitz bound]\label{lem:Lipschitz2}
Let $u$ be a classical solution of \eqref{FBP} in $B_1\setminus B_{\d}$, satisfying (A). Then
\begin{equation*}
    |\nabla u|(x)\leq C \quad \text{for all } x\in B_{1/2}\setminus B_{2\d}
\end{equation*}
for some numerical constant $C>0$. 
\end{lemma}
\begin{proof}
Let $x \in \left(B_{1/2}\setminus B_{2\d}\right)^+$. Denote by $D^M(x)$ the largest disk centered at $x$ and contained in the positive phase $(B_1\setminus B_{\d})^+$. If $D^M(x) = B_r(x)$ touches $F(u)$, then by the standard argument via the Hopf Lemma (\cite{AC}) one can show that $u(x)\leq c r$, so that by the Harnack inequality and the classical gradient estimate we obtain $|\nabla u (x)| \leq C$. If not, then $D^M(x)$ touches $\de B_{\d}$, and we can iteratively build a finite Harnack chain $\{B_{2r/3}(x_{i})\}_{i=1}^l$ with $$x_1=x, \quad|x_{i}| = |x| \quad \text{and} \quad  x_{i+1}\in \de B_{r/2}(x_{i})$$
with length $l$ universal, for which eventually $D^M(x_l)=B_{\rho}(x_l)$ touches $F(u)$. By the argument above $u(x_l)\leq c \rho$ and since $\rho \leq r$, we have
\[
u(x) \leq c' u(x_l) \leq c'c \rho \leq c'c r,
\]
whence the Lipschitz bound follows once again.
\end{proof}

Next, we will show that $u$ also has non-degenerate growth away from the free boundary as well as from the center of the annulus. 
\begin{lemma}[Non-degeneracy]\label{lem:non-degeneracy2}
Let $u$ be a classical solution of \eqref{FBP} in $A=B_1\setminus B_{\d}$, satisfying condition (A). If $x_0\in F(u)$ and $B_{\rho}(x_0)$ is the largest disk contained in $A^+(u)$, then
\begin{equation}\label{eq:nondeg1}
    \sup_{B_r(x_0)} u  = \max_{\de B_r(x_0)} u \geq \frac{1}{4\pi}r \quad \text{for all} \quad 0<r< \rho.
\end{equation}
Furthermore, there exist absolute positive constants $c, \tilde{c}$ such that 
\begin{equation}\label{eq:nondeg2}
    \max_{\de B_r} u \geq c r \quad \text{for all} \quad \tilde{c} \d \log(1/\d)<r<1.
\end{equation}
\end{lemma}
\begin{proof}
The first part \eqref{eq:nondeg1} of this lemma follows as in \cite[Proposition 3.4]{JK}. For the second part \eqref{eq:nondeg2}, we use the same idea. 
Define
$$G_r(x) := \frac{1}{2\pi} \log(|x|/r).$$
to be the Green's function of $B_r$ with pole at the orign. Noting that 
$u = 0$ and $u_{\nu} = -1$ on $F(u)$, we apply Green's formula in $D:=(B_r\setminus B_{2\d})^+$
\begin{align*}
0 &=\int_{D} \D G_r u - G_r \D u ~dx =  \int_{\de D} (G_r)_{\nu} u - G_r u_{\nu} ~d\mathcal{H}^{1} = \\
&=\dashint_{\de B_r} u ~d\mathcal{H}^{1} + \int_{\de D \cap F(u)} \frac{1}{2\pi} \log (|x|/r)  ~d\mathcal{H}^{1} - \int_{\de B_{2\d}} \left( \frac{u}{4\pi \d} + \frac{u_{\nu}}{2\pi} \log\frac{2\d}{r}\right)d\mathcal{H}^{1}.
\end{align*}
By Lemma \ref{lem:Lipschitz2} we know that $u$ satisfies a universal Lipschitz bound in $B_{1/2}\setminus B_{2\d}$:
\[
|\nabla u| \leq C \quad \text{and thus} \quad u\leq 2C\d \quad \text{on} \quad  \de B_{2\d}.
\]
Therefore, 
\begin{align}\label{lem:non-degeneracy2_ineq}
 \dashint_{\de B_r} u  ~d\mathcal{H}^{1} \geq  \int_{\de D \cap F(u)} \frac{1}{2\pi} \log (r/|x|)  ~d\mathcal{H}^{1} - c'\d \log(1/\d).
\end{align}
Since 
\[
\log (r/|x|) \geq \log 2 \quad \text{on} \quad \de D\cap F(u) \cap B_{r/2} \quad 
\]
and given that $\de D\cap F(u) \cap B_{r/2}$ contains at least two curves that connect $\de B_{2\d}$ to $\de B_{r/2}$ implying
\[
\mathcal{H}^{1} (\de D\cap F(u) \cap B_{r/2}) \geq 2 (r/2 - 2\d) = r - 4\d,
\]
we see from \eqref{lem:non-degeneracy2_ineq} that 
\[
 \dashint_{\de B_r} u  ~d\mathcal{H}^{1} \geq  \int_{\de D \cap F(u) \cap B_{r/2}} \frac{1}{2\pi} \log (r/|x|)  ~d\mathcal{H}^{1} - c'\d \log(1/\d)\geq  \frac{\log 2}{2\pi} r -  c''\d \log(1/\d).
\]
Hence,
\[
 \max_{\de B_r} u  \geq \dashint_{\de B_r} u  ~d\mathcal{H}^{1} \geq cr \quad \text{when} \quad 1>r>\tilde{c} \d \log(1/\d).
\]
\end{proof}
 
We shall now prove the following key theorem identifying limits of sequences of solutions satisfying the assumption (A) as the inner radius goes to zero and the outer increases to infinity.

\begin{theorem}\label{keyprop} Let $\{u_k\}$ be a sequence of classical solutions of \eqref{FBP} in $B_{R_k}\setminus B_{\d_k}$, satisfying assumption (A), where $R_k \nearrow \infty$ and $\d_k\searrow 0$. Then a subsequence converges uniformly on compact subsets of $\R^2\setminus \{0\}$ to, up to rotation, either the one-plane solution $P$, the two-plane solution $TP_a$ for some $a> 0$, or the wedge $W_1$.
\end{theorem}
\begin{proof}
From Lemma \ref{lem:Lipschitz2} we see that the family $\{u_k\}$ is uniformly Lipschitz continuous on any compact subset of $\R^2\setminus \{0\}$, thus a subsequence converges uniformly to some Lipschitz continuous function $u: \R^2\setminus \{0\} \to \R$. In the next several steps, we shall prove that $u$ could only be one of the possibilities stated. 

\emph{Step 1.} Let us first argue that $0$ is a removable singularity of $u$:
\[
\lim_{x\to 0} u(x) = 0.
\]
Indeed, the universal Lipschitz bound in $B_{R_{k/2}}\setminus B_{\d_k}$ and the fact that $F(u_k)$ intersects $B_{\d}$ for any fixed $\d>0$, imply
\[
0\leq u_k(x) \leq C \d \quad \text{for all}\quad x\in \de B_{\d},
\]
whence $0 \leq u \leq C\d$ on $\de B_{\d}$, so that  $\lim_{x\to 0} u(x)$ exists and is $0$. We may now extend $u$ as a Lipschitz continuous function on all of $\R^2$. Moreover, the origin $0\in F(u)$, for if some disk $B_{2\d}(0)\subseteq \{u = 0\}$, then $\max_{\de B_{\d}}u_{k}\to 0$ as $k\to \infty$ which would contradict the universal non-degeneracy bound \eqref{eq:nondeg2} of Lemma \ref{lem:non-degeneracy2}:
\[
\max_{\de B_{\d}(0)} u_k \geq c \d \quad \text{given that} \quad  \d > \tilde{c}\d_k \log(1/\d_k).
\]

\emph{Step 2.} We will show that $u$ is a variational solution in all of $\R^2$. Let $\phi \in C^{\infty}_0(\R^2; \R^2)$. Denote for $\d >0$
\[
 L_{\d}[u](\phi) := \int_{\R^2 \setminus B_{\d}} (|\nabla u|^2 + \chrc{u>0}) \text{div}\phi - 2 \nabla u D\phi (\nabla u)^T~dx.
\]
In order to verify the Definition \ref{def:varsolution}, it suffices to show
\begin{equation}\label{keyprop_eq1}
\lim_{\d \to 0} L_{\d}[u](\phi) = 0.
\end{equation}
Applying Proposition \ref{prop_limitofsolns} in $\O=(B_{\d})^c$ we see that $\chrc{\{u_k>0\}}\to \chrc{\{u>0\}}$ in $L^1_{\text{loc}}((B_{\d})^c)$ and $\nabla u_k \to \nabla u$ a.e. in $(B_{\d})^c$ with $|\nabla u_k| \leq C $, so that by the Dominated Convergence Theorem 
\begin{equation}\label{keyprop_eq2}
L_{\d}[u](\phi) = \lim_{k\to \infty} L_{\d}[u_k](\phi).
\end{equation}
Now taking into account that $u_k$ is a classical solution and that 
\[
 (|\nabla u_k|^2 + \chrc{u_k>0}) \text{div}\phi - 2 \nabla u_k D\phi (\nabla u_k)^T = \text{div}\left((|\nabla u_k|^2+1)\phi -2(\phi \cdot \nabla u_k)\nabla u_k \right) \quad \text{in } \{u_k>0\}
\]
we obtain from the Divergence Theorem and the Lipschitz bound $|\nabla u_k| \leq C$ on $\de B_{\d}$:
\begin{align}\label{keyprop_eq3}
& \int_{\R^2 \setminus B_{\d}} (|\nabla u_k|^2 + \chrc{u_k>0}) \text{div}\phi - 2 \nabla u_k D\phi (\nabla u_k)^T~dx =  \int_{F(u_k) \setminus B_{\d}} \left((1+1)\phi\cdot \nu -2\phi \cdot \nu \right) d\mathcal{H}^1 + \notag \\
 & +  \int_{\de B_{\d}\cap\{u_k>0\}} \left((|\nabla u_k|^2+1)\phi\cdot \nu  -2(\phi \cdot \nabla u_k)\nabla u_k\cdot\nu \right) d\mathcal{H}^1 = O(\d \|\phi\|_{L^{\infty}}).
\end{align}
Putting together \eqref{keyprop_eq2} and \eqref{keyprop_eq3}, we establish \eqref{keyprop_eq1}. 

\emph{Step 3.} Let us show that $u$ is non-degenerate. First, we claim that $u$ has non-degenerate growth away from the origin:
\begin{equation}\label{eq:nondeg3}
\sup_{B_r} u = \max_{\de B_r} u \geq cr \quad \text{all} \quad r>0
\end{equation}
for some absolute constant $c>0$. This follows from the uniform convergence of $u_k$ to $u$ on $\de B_r$ and the bound \eqref{eq:nondeg2} of Lemma \ref{lem:non-degeneracy2} (note that a fixed radius $r$ is eventually greater than $\tilde{c}\d_k \log(1/\d_k)$ as $\d_k\to 0$).

Now let $p\in F(u)$ be an arbitrary free boundary point: we will show that $u$ has non-degenerate growth away from $p$. By Lemma \ref{lem:non-degeneracy2} and Proposition \ref{prop_limitofsolns} we know that $u$ is non-degenerate at $p$ certainly up to scale $|p|$:
\[
\sup_{B_r(p)} u \geq c r \quad \text{for}\quad 0<r<|p|.
\]
so that by potentially halving the constant $c'=c/2$, we have 
\[
\sup_{B_r(p)} u \geq c' r \quad \text{for}\quad 0<r\leq 2|p|.
\]
For the remaining scales $r> 2|p|$, we have that $B_{r}(p)\supseteq B_{r/2}(0)$, thus \eqref{eq:nondeg3} implies
\[
\sup_{B_r(p)} u \geq \sup_{B_{r/2}} u \geq c r/2 = c'r \quad \text{for}\quad r>2|p|.
\]

\emph{Step 4.} We will show that $u$ is a viscosity solution in all of $\R^2$, as well. 
As Lemma \ref{lem:non-degeneracy2} indicates that $u_k$ are non-degenerate in compact subsets of $\R^2\setminus \{0\}$, Proposition \ref{prop_limitofsolns} implies that the limit $u = \lim_{k} u_k$ is certainly a viscosity solution in $\R^2\setminus \{0\}$.

Let us show that $u$ satisfies the viscosity solution condition at the origin, as well. Assume that at $0$ there is a touching disk $D\subseteq \{u=0\}$ to $F(u)$ from the zero phase, and let $e_1$ be the outer unit normal to $\de D$ at $0$. Because $u$ has non-degenerate growth at the origin \eqref{eq:nondeg3}, we deduce by \cite[Lemma 11.17]{CafSalsa}
\[
u(x) = s x_1^+ + o(|x|) \quad \text{for some } s>0 \text{ in non-tangential regions of } D^c.
\]
On the other hand, by the previous step we can apply Proposition \ref{prop_blowups} to conclude that $u$ necessarily blows up to the one-plane solution $P$. Thus, we confirm that $s=1$, i.e. the viscosity subsolution condition at $0$ is satisfied. 

Suppose now that $D\subseteq \{u>0\}$ is a touching disk to $F(u)$ at $0$ from the positive phase and that $e_1$ is the inner unit normal to $\de D$ at $0$. Then \cite[Lemma 11.17]{CafSalsa} states that
\[
u(x) = s x_1^+ + o(|x|) \quad \text{for some } s>0 \text{ in non-tangential regions of } D.
\]
Fix $\eps>0$ small. Since $u$ is a non-degenerate variational solution, it blows up at the origin to $P$ or $W_s$  for some $s>0$, so that there is a small enough $r_0=r_0(\eps)$ such that  
\[
|u(x) - s x_1^+| \leq s\eps r_0/2 \quad \text{in} \quad B_{r_0}\cap D.
\]
and 
\[
F(u)\cap B_{r_0} \subseteq \{|x_1| < \eps r_0/2\}.
\]
By the uniform convergence of $u_k$ to $u$ in $\bar{D}_{r_0}\setminus B_{\eps r_0}$, we see that for all large enough $k$
\[
|u_k(x) - sx_1^+|< s\eps r_0 \quad \text{in} \quad (\bar{D}_{r_0}\setminus B_{\eps r_0})\cap D
\]
and since $F(u_k) \to F(u)$ in the Hausdorff distance in $\bar{D}_{r_0}\setminus B_{\eps r_0}$, we have
\[
F(u_k)\cap (B_{r_0} \setminus B_{\eps r_0}) \subseteq \{|x_1| < \eps r_0 \}
\]
for all large enough $k$. Define the domain
\[
U_t = \{x \in \bar{D}_{r_0} \setminus B_{\eps r_0}: x_1 >  \eps r_0 + t \eta(x_2) r_0\}
\]
where $0\leq \eta(x_2) \leq 1$ is a smooth bump function supported in $(\eps r_0, 2r_0/3)$ with $\eta(x_2) = 1$ for $x_2 \in [r_0/3, r_0/2]$. Then $U_0\subseteq \{u_k>0\}$, but for some $-3\eps<t_0<0$, $U_{t_0}$ will touch $F(u_k)$ at a point $p\in F(u_k)\cap \{\eps r_0<x_2 < 2r_0/3\}.$ Now we can define a harmonic function $v$ in $U_{t_0}$ with boundary values
\[
v(x) = \left\{\begin{array}{rcl} s x_1 - s\eps r_0 & \text{on} & \de B_{r_0}\cap \{x_1>\eps r_0\}\\ 0 &  \text{on} & (B_{r_0}\setminus B_{\eps r_0})\cap \{x_1 = \eps r_0 + t_0  \eta(x_2) r_0 \}.\end{array}\right.
\]
By the maximum principle, $v \leq u_k$ in $U_{t_0}$. Furthermore, the viscosity supersolution condition for $u_k$ (recall Definition \ref{def:visc.super}) implies that near $p$ in $U_{t_0}$
\[
v(x) \leq u_k(x) = \tilde{s} \langle x-p, \nu(p)\rangle + o(|x-p|) \quad \text{for some } \tilde{s}\leq 1,
\]
where $\nu(p)$ is the inner unit normal to $\de B_{t_0}$ at $p$. On the other hand, a standard perturbation argument gives 
\[
v_{\nu}(p) = s + O(\eps).
\]
Since $\eps$ is arbitrary we conclude that $s\leq \tilde{s}\leq 1$, and thus the viscosity supersolution condition holds at $0$, as well.

\emph{Step 5.} Having established that $u$ is both a variational and a viscosity solution in all of $\R^2$,
we can argue as in Proposition \ref{prop:Voutcomes} to show that either $u=P$, or $u = TP_a$ for some $a> 0$, or $u=W_1$, or that $|\nabla u|<1$ in $\{u>0\}$ with $u$ blowing down to a wedge solution $W_s$ for some $0<s\leq 1$. We shall prove that this last possibility doesn't actually arise.

First we follow analogous reasoning to Theorem \ref{thm:mainglobal} to show that any such $u$ with $|\nabla u| <1$
in the positive phase is a double hairpin solution. 
Denote the two components of $F(u_k)$ connecting $\de B_{\d_k}$ to $\de B_{R_k}$ by $\g_1$ and $\g_2$, and assume that we are in that last scenario: $u$ blows down to a wedge $W_s$ and $|\nabla u| <1$ everywhere in the positive phase. 
Let us first show that at every $p \in F(u)\setminus \{0\}$ $u$ blows up to the one-plane solution $P$, thereby concluding that $F(u)$ is smooth at $p$ by the classical Alt-Caffarelli result. 
Suppose on the contrary that $u$ blows up to a wedge $W_s=s|x_1|$ for some $0<s\leq 1$. Then arguing as in the proof of Lemma \ref{lem:almostsmooth} we can show that for any $r>0$ and $k$ large enough, $B_r(p)\cap \{u_k=0\}$ consists of exactly two connected components. Let $S$ be the shortest segment connecting them. By the flux balance Lemma \ref{lem:fluxbalance} it is easy to see that the two endpoints of $S$ have to belong to separate connected components of $F(u_k)$. But then $\g_1$, $S$, $\g_2$ and $\de B_{2\d_k}$ enclose a subregion $D$ of $B_{2\d_k}$, where $u_k$ satisfies the universal Lipschitz bound of Lemma \ref{lem:Lipschitz2} and for which $$\mathcal{H}^1(\de D\setminus F(u_k)) \leq \mathcal{H}^1(S) +\mathcal{H}^1(\de B_{2\d_k})\leq r +4\pi \d_k.$$
On the other hand, $\mathcal{H}^1(F(u_k)\cap \de D)\geq |p|-\d_k$ and we can invoke the flux balance lemma again to reach a contradiction, taking $r$ and $\d_k$ small enough.

Thus, $F(u)$ is smooth everywhere, but possibly at the origin. It is not hard to see that the topological assumption (A), which the $u_k$ satisfy, plus the fact that $F(u_k) \to F(u)$ on compact subsets of $\R^2\setminus\{0\}$ precludes $F(u)\setminus \{0\}$ from having a closed curve as a connected component. 
Now we are exactly in a position to invoke \cite[Lemma 7.5 and 7.6]{JK} and conclude that $F(u)$ is smooth at the origin, as well. Hence, $u$ is a global \emph{classical solution} of \eqref{FBP} that blows down to a wedge $W_s$ and satisfies $|\nabla u| <1$ in $\{u>0\}$. Therefore, arguing as in Proposition \ref{prop:Voutcomes}(c) we conclude that $u=H_a$, the double hairpin solution of separation $a$ between the two hairpin curves $\g_{\pm}$.

At this point our argument diverges from the one of Theorem \ref{thm:mainglobal} because we wish to rule
out the possibility $u = H_a$. 
Denote by $p_1, p_2$ the points where $\gamma_1$ and $\g_2$ connect to $\de B_{\d_k}$, respectively, and let $\alpha$ be the arc of $\de B_{\d_k}$ between $p_1$ and $p_2$ that belongs to $\de \{u_k>0\}$. Then $\G = \gamma_1\cup \alpha \cup \g_2$ is a piecewise smooth Jordan arc.  Let $p$ be the center of the double hairpin solution $u$ and consider a large disk $B_R(p)$, with $R\gg a \gg \d_k$ such that $0\in B_{R/2}(p)$. Fix $\eps>0$. By the Hausdorff distance convergence of $F(u_k)$ to $F(u)$ locally, we have that 
\[
\G \cap B_{R}(p) \subseteq (\mathcal{N}_{\eps}(\g_+) \cap B_R(p)) \sqcup  (\mathcal{N}_{\eps}(\g_-)\cap B_R(p) )
\]
and 
\[
(\g_+ \sqcup \g_-) \cap B_{R}(p) \subseteq \mathcal{N}_{\eps}(\G) \cap B_{R}(p)
\]
for all large enough $k$. Thus, there exists a straight line segment $\beta \subseteq \{u_k>0\}$ such that $\bar{\beta}$ realizes the distance of size $\approx a$ between the disjoint compact sets $\G_+ :=\G \cap  \overline{(\mathcal{N}_{\eps}(\g_+) \cap B_R(p))}$ and $\G_-:=\G \cap  \overline{(\mathcal{N}_{\eps}(\g_-) \cap B_R(p))}$. Denote the two endpoints of $\bar{\beta}$ by $p_+\in \G_+$ and $p_-\in \G_-$ and let $\G(p_-,p_+)$ denote the closed segment of the Jordan arc $\G$ between the two points $p_+$ and $p_-$. Without loss of generality we may assume that it is $\G_-$ that intersects $B_{2\d_k}(0)$, so that surely $p_+\notin B_{2\d_k}(0)$. Now, by the Jordan Curve Theorem, $\G(p_-,p_+)\cup \beta$ encloses a bounded connected, piecewise smooth domain $\O\subseteq \{u_k>0\}$. Let $\O_0$ be the connected component of $\O\setminus B_{2\d_k}$ whose boundary contains $p_+$. Then $\O_0$ is a bounded, piecewise smooth domain that is a subset of the positive phase $\{u_k>0\}$ and $u_k$ satisfies the universal Lipschitz bound $|\nabla u_k| \leq C$ over $\overline{\O}_0\subseteq (B_{2\d_k})^c$ of Lemma \ref{lem:Lipschitz2}. However,
$$\mathcal{H}^1(\de \O_0 \cap F(u_k)) \geq R  \quad \text{while}\quad \mathcal{H}^1(\de \O_0 \setminus F(u_k)) = O(\d_k + a + \eps).
$$
so that we reach a contradiction by invoking the flux balance lemma.

\end{proof}

Having established the classification of limits above, we will now show by a compactness argument that a solution $u$, satisfying (A) in an annulus $B_{1}\setminus B_{\d}$, has a free boundary that is $\eps$-flat on all scales $C\d<r <1/C$ for some $C=C(\eps)$. 

\begin{lemma}[Flatness]\label{lem:flatness} For every $\eps>0$, there exist $0<\d_0\ll 1$ and $C>1$ such that if $u$ is a classical solution of \eqref{FBP} in $A=B_{1}\setminus B_{\d}$ with $0<\d <\d_0$ and $u$ satisfies the topological assumption (A), then for every $r\in (C\d, 1/C]$
\begin{equation*}
    |u\chrc{T}(\rho x) - P(x)| \leq \eps r \quad \text{in} \quad B_{r}\setminus B_{C\d}, \quad \text{for some rotation }  \rho = \rho_r,
\end{equation*}
where $T$ is the connected component of $(B_{1/C}\setminus B_{\d})^+(u)$, whose closure intersects $\de B_{\d}$.
\end{lemma}
\begin{proof}
We  argue by contradiction. Fix $\eps>0$ and assume that the statement of the lemma is false. Then for any sequence $\d_k \searrow 0$ and any $C_k \nearrow \infty$, there exists a sequence $\{v_k\}_k$ of counterexamples: namely, there exist classical solutions $v_k$ in $B_{1}\setminus B_{\d_k}$ satisfying assumption (A), such that  
if $\tilde{T}_k$ denotes the connected component of $(B_{1/C_k}\setminus B_{\d_k})^+(v_k)$ whose closure intersects $\de B_{\d_k}$, then for all rotations $\rho$
\begin{equation}\label{lemmaFlat_ctrex}
\|v_k\chrc{\tilde{T}_k}(\rho x) - P(x)\|_{L^{\infty}(B_{r_k}\setminus B_{C_k\d_k})} > \eps r_k \quad \text{for some} \quad r_k\in (C_k \d_k, 1/C_k].
\end{equation}
Choose the sequences $\{\d_k\}_k$ and $\{C_k\}_k$ in such a way that
\[
\d_k \searrow 0, \quad C_k  \nearrow \infty \quad \text{with} \quad \d_k C_k^2 \to 2/3 \quad \text{as}\quad k\to \infty
\]
and consider the rescales:
\[
u_k(x) = v_k(r_k x)/r_k \quad \text{for} \quad x\in B_{1/r_k}\setminus B_{\d_k/r_k}.
\]
Denote $T_k:=r_k^{-1} \tilde{T}_k$. Since $C_k\d_k/r_k \geq C_k^2 \d_k > 1/2$ for all large $k$, \eqref{lemmaFlat_ctrex} implies that for all rotations $\rho$,
\begin{equation}\label{lemmaFlat_ctrexRescale}
\|u_k\chrc{T_k}(\rho x) - P(x)\|_{L^{\infty}(B_{1}\setminus B_{1/2})} \geq \|u_k\chrc{T_k}(\rho x) - P(x)\|_{L^{\infty}(B_{1}\setminus B_{C_k\d_k/r_k})} > \eps.
\end{equation}
On the other hand, the solutions $u_k$ satisfy hypothesis (A) in the annulus $(B_{1/r_k}\setminus B_{\d_k/r_k})\supseteq B_{C_k}\setminus B_{1/C_k}$, hence Theorem \ref{keyprop} states that the $u_k$ converge uniformly on compacts subsets of $\R^2\setminus \{0\}$ to some rotation of either $P$, $W_1$ or $TP_a$, for some $a> 0$. Let us show that each leads to a contradiction with \eqref{lemmaFlat_ctrexRescale}.

\noi
\emph{Case 1.} Assume that $u_k(\rho x) \to P(x)$ for some rotation $\rho$. Then by Alt-Caffarelli regularity, for all large $k$, \mbox{$A_k^+:=(B_1\setminus B_{1/2})^+(u_k)$} is the epigraph of a Lipschitz function in the annulus $(B_1\setminus B_{1/2})$ and, thus, has a single connected component. But since each connected component of the non-empty $T_k \cap (B_{1}\setminus B_{1/2})$  is a connected component of $A_k^+$, it has to be that 
\[
T_k \cap (B_{1}\setminus B_{1/2}) = A_k^+.
\]
Hence for all large $k$,
\[
\|u_k\chrc{T_k}(\rho x) - P(x)\|_{L^{\infty}(B_{1}\setminus B_{1/2})}  =  \|u_k(\rho x) - P(x)\|_{L^{\infty}(B_{1}\setminus B_{1/2})} \leq \eps,
\]
which contradicts \eqref{lemmaFlat_ctrexRescale}.

\noi
\emph{Case 2.} Assume that $u_k(\rho x) \to W_1(x)$ for some rotation $\rho$. Denote by  $\g_1$, $\g_2$ the two components  of $F(u_k(\rho x))$ and let $\g_i$ attach at $p_i\in \de B_{\d_k/r_k}$, $i=1,2$.  Let $\G=\G_k$ be the simple Jordan arc that is the union of  $\g_1$, $\g_2$ and the circular arc $\widehat{p_1p_2}$ of $\de B_{\d_k/r_k}$ that belongs to $\de \{u_k>0\}$, and let $\G(r)$, for $r>\d_k/r_k$ be the connected component of $\G\cap B_r$ that contains $p_1$ and $p_2$. Set $s_k=1/(C_kr_k)$, so that we have $$1<s_k<1/(\d_k C_k^2) < 2.$$
Denote by $\a_{\pm}(r)$ the two circular arcs of $\de B_r\cap \{|x_1|<\eps/2 \}$. 
By the convergence of $u_k(\rho x)$ to $W_1(x)$ on compact subsets, we see that the two endpoints of $\G(r)$ belong to  $\a_{+}(r)\cup \a_-(r)$, for every $1/2\leq r\leq 3$. 

Claim that the two endpoints of $\G(s_k)$  belong to distinct arcs $\a_{\pm}(s_k)$. Assume not, i.e. that they belong to the same arc, say $\a_+(s_k)$. By applying the argument of Case 3 of Lemma \ref{lem:flattypes}, necessarily the two endpoints of $\G(r)$ belong to $\a_+(r)$, for every $s_k\leq r \leq 3$. The lemma further implies that $\G(3)$ doesn't intersect $\alpha_-(r)$ for any $s_k\leq r\leq 3$ and no component of $\G\cap B_3$, other than $\G(3)$, can intersect $\alpha_+(r)$ for any $s_k\leq r \leq 2$. By the convergence of $F(u_k)$ to $\{x_1=0\}$ in $B_4\setminus B_{1/2}$, we see that $\G \cap B_2$ has connected components that attach on $\alpha_-(2)$: let $F_-$ denote their union of all of those. Then the dist$(\overline{\G(2)}, \overline{F_-}) <\eps$ and is realized at some points $p \in \overline{\G(2)}$ and $q\in \overline{F_-}$ so that the straight line segment $pq\subseteq \{u_k>0\}$. Thus the segment $\G(p,q)$ of $\G$ between $p$ and $q$ and the straight line segment $pq$ enclose a piecewise smooth Jordan domain $D$ of positive phase that has diam$(D)>2$. Applying the flux balance Lemma \ref{lem:fluxbalance} to the connected component of $D\setminus B_{\eps}$ intersecting $\de B_1$, we reach the desired contradiction.

Now, given that $\G(s_k)\subseteq \de T_k$ we see that 
$$\rho^{-1}T_k\subseteq \{x_1>-\eps/2\} \quad \text{and}\quad F(u_k\chrc{T_k}(\rho x)) \cap (B_{1}\setminus B_{1/2}) \subseteq \{|x_1|<\eps/2\}$$
so that
\begin{align*}
& \|u_k\chrc{T_k}(\rho x) - P(x)\|_{L^{\infty}(B_{1}\setminus B_{1/2})} \leq \|u_k(\rho x) - P(x)\|_{L^{\infty}((B_{1}\setminus B_{1/2})\cap \{x_1>-\eps/2\})} \\
& \leq  \eps/2 + \|u_k(\rho x) - W(x)\|_{L^{\infty}((B_{1}\setminus B_{1/2})\cap \{x_1>-\eps/2\})} \leq \eps
\end{align*}
for all large $k$, as $u_k(\rho) \to W$ uniformly on compacts subsets of $\R^2\setminus \{0\}$. However, this contradicts \eqref{lemmaFlat_ctrexRescale}.

\noi
\emph{Case 3.} Suppose that $u_k(\rho x) \to TP_a(x)=x_1^+ + (-x_1-a)^+$ for some rotation $\rho$. Then for all large $k$, by Alt--Caffarelli regularity $(B_{s_k}\setminus B_{\d_k/r_k})^+(u_k(\rho x))$ has a connected component $D_{k}^+$ which is the subgraph of a Lipschitz function, with $D_k^+\subseteq B_{s_k}\cap \{x_1<-a/2\}$. In particular, $\G(s_k)\cap \overline{D_k^+} = \emptyset$.

Redefine $\alpha_{\pm}(r)$ to be the two circular arcs of $\de B_r\cap\{|x_1|<\min(\eps/2,a/4)\}$. As in the previous Case 2, it has to be that, for large enough $k$, the endpoints of $\G(r)$ belong to $\a_-(r)\cup\alpha_{+}(r)$, $1/2\leq r\leq 3$. Applying the Alt-Caffarelli regularity, we further see that each of the two endpoints of $\G(s_k)$ has to belong to a distinct arc $\alpha_{\pm}(s_k)$. Thus,
$$\rho^{-1}T_k\subseteq \{x_1>-\min(\eps/2,a/4)\} \quad \text{and}\quad F(u_k\chrc{T_k}(\rho x))\cap (B_{1}\setminus B_{1/2}) \subseteq \{|x_1|<\min(\eps/2,a/4)\}$$
so that  by the uniform convergence of $u_k(\rho x)$ to $TP_a(x)$ on compact subsets of $\R^2\setminus \{0\}$
\begin{align*}
& \|u_k\chrc{T_k}(\rho x) - P(x)\|_{L^{\infty}(B_{1}\setminus B_{1/2})} \leq \|u_k(\rho x) - P(x)\|_{L^{\infty}((B_{1}\setminus B_{1/2})\cap \{x_1>-\min(\eps/2,a/4)\})} = \\
& = \|u_k(\rho x) - TP_a(x)\|_{L^{\infty}((B_{1}\setminus B_{1/2})\cap \{x_1>-\min(\eps/2,a/4)\})} \leq \eps
\end{align*}
for all large $k$, contradicting \eqref{lemmaFlat_ctrexRescale}. This completes the proof of the lemma.
\end{proof}

Now that flatness of the free boundary on all scales $C\d<r<1/C$ has been established for solutions that satisfy condition (A), the proof of the effective removable singularity Theorem \ref{thm:removable} reduces to applying the analogous weaker result proved in \cite{JK}, which has flatness as an extra hypothesis. For the reader's convenience,  we provide its statement below and a sketch of the proof (for the full proof we refer to the discussion in \cite[Section 9]{JK}).

\begin{theorem}[``Flat with a small hole implies Lipschitz''] \label{thm:effremov-weak}
There exist numerical constants $\eps_0>0$ and $c>0$ such that for every $0<\eps<\eps_0$ and every $d>0$, the following holds: if $u$ is a classical solution in the annulus $A=B_1\setminus B_{d}$, which satisfies assumption (A) and whose free boundary is flat on all scales in the sense that
\begin{equation}\label{eq:flatness-thm:effremov}
    |u(\rho x) - P(x)| \leq \eps r \quad \text{in} \quad B_{r}\setminus B_{d}, \quad \text{for some rotation }  \rho = \rho_r \quad \forall r\in (d,1),
\end{equation}
then there is a function $g:[-1,1]\to \R$ satisfying $g(0) = 0$ and $|g'(x)| \le c\e
$ for all $x\in [-1,1]$ such that after rotation
\[
\{z\in \C: cd < |z| < 1/c\} \cap A^+(u) = 
\{z\in \C: cd < |z| < 1/c\} \cap \{z = x_1+ix_2: x_1 > g(x_2)\}.
\]
\end{theorem}

\begin{proof}  Take $\eps_0$ be small enough so that the Alt-Caffarelli theory says that if $F(u)$ is $\eps$-flat in $B_1$ for $\eps<\eps_0$, then $F(u)$ is a Lipschitz graph of Lipschitz norm $C\eps$ in $B_{1/2}$.

Let $u$ be as in the statement of the theorem. It is not hard to see that the flatness condition \eqref{eq:flatness-thm:effremov} implies that
\begin{equation}\label{thm:effremov-eqn1}
||\nabla u| - 1| \le C\e \quad \text{for all} \quad  z\in \{2d \leq |z| \leq 1/2\} \cap A^+(u).
\end{equation}
In particular, the length of $\nabla u$ is 
very close to $1$, but the issue is that the direction
of $\nabla u$ may change:  the possibility we still need to 
rule out for $F(u)$ is a logarithmic spiral.

Let $e_1$ be a unit vector in the direction of $x_1$. Denote by $A_k$ the annuli $A_k:=B_{r_{k+1}}\setminus B_{r_{k-1}}$, where $r_k = 2^k d$, for $k= 1, 2, 3 \ldots, N$, with $N=\lfloor \log_2 (1/d)\rfloor$, so that $r_N\in (1/2,1)$. By the Alt-Caffarelli theory applied in $(B_{r_{k+2}}\setminus B_{r_{k-2}})$, $k=2,\ldots N-2$, we see that $F(u)\cap A_k$ is the graph of a $c\eps$--Lipschitz function in the direction of $\rho_k^{-1}e_1$, for some rotation $\rho_k$. Because the dyadic annuli $A_k$ are overlapping ($A_k\cap A_{k+1} = B_{r_{k+1}}\setminus B_{r_k}$), it must be that
\begin{equation}\label{eqn:succrot}
|\rho_{k}^{-1}e_1 - \rho_{k-1}^{-1}e_1| \leq c'\eps \qquad k=2,\ldots, N-2. 
\end{equation}
Also, the fact that $F(u)$ is a graph with a small Lipschitz norm in each annulus $A_k$ ensures that $\de B_{r_k}\cap F(u)$ consists of exactly two points for every $k=2,\ldots, N-2$. Denote by $p$ the intersection of $\de B_{r_{1}}$ with $F(u)$ such that $\rho_{1}^{-1} p$ is contained in the lower half-plane $\{\text{Im } z<0\}$. Now define $\tilde{u}$ to be the harmonic conjugate of $u$ in the simply connected region $A^+(u)$, choosing a normalization such that $\tilde{u}(z)\to -|p|$ as $z\to p$. Set $U=u+i\tilde{u}$.  Thus, \eqref{thm:effremov-eqn1} is strengthened to 
\[
|U'(z) - e^{i\theta_k}|\leq C\eps \quad \text{when} \quad z\in A_k^+(u), \qquad k= 2, 3 \ldots, N-2
\]
where, according to \eqref{eqn:succrot}, the angles $\th_k$ satisfy
\begin{equation}\label{eqn:angles}
|e^{i\theta_k} - e^{i\theta_{k-1}}| \leq c' \eps \qquad k= 2, 3 \ldots, N-2.
\end{equation}
Now one can show by induction that 
\begin{equation}\label{eqn:integrate}
|U(z) - e^{i\th_k}z| \leq C'\eps|z| \quad \text{for all } z\in A_k^+(u), \quad k= 2, 3 \ldots, N-2.
\end{equation}
This estimate is sufficient to establish the injectivity of $U$ on $\tilde{A}:=(B_{r_{N-1}}\setminus B_{2d})^+(u)$ for suitably small $\eps >0$. Indeed, if $U(z_1) = U(z_2)$ for some $z_{1}, z_2 \in \tilde{A}$ with $|z_2|\geq |z_1|$, then
\[
1\leq \frac{|z_2|}{|z_1|} \leq \frac{|U(z_2)|/(1-C'\eps)}{|U(z_1)|/(1+C'\eps)} = \frac{1+C'\eps}{1-C'\eps} < 2.
\]
Thus, $z_1$ and $z_2$ belong to some $(A_k\cup A_{k+1})^+(u)$. However, \eqref{eqn:integrate} and \eqref{eqn:angles} imply that $U$ is injective on $(A_k\cup A_{k+1})^+(u)$ for small enough $\eps$, hence it must be that $z_1=z_2$. 

Therefore, for suitably small $\eps>0$, $U$ maps $\tilde{A}$ biholomorphically onto its image contained in the right 
half-plane. Moreover, the image contains a half annulus $D_R$ in the right half-plane. 
\[
U\big((B_{r_{N-1}}\setminus B_{2d})^+(u)\big)\supseteq D_R:=\{\z\in \C: \Re \z >0, ~ 2d(1+c\eps) <|\z| < r_{N-1} (1-c\eps)\}
\]
and $U$ extends smoothly to the boundary of $\tilde{A}$, sending $F(u)\cap \de \tilde{A}$ into the imaginary axis. Thus, in $D_R$ one can define the holomorphic inverse
\[
\Phi = U^{-1}: D_R \to \tilde{A}
\]
which extends smoothly to the boundary, with $\Phi$ mapping $\de D_R\cap \{\Re \z = 0\}$ into $F(u)$. Consider its logarithmic derivative in $D_R$
\[
\Phi'= e^{h + i \tilde{h}}: D_R \to \tilde{A}.
\]
Since $|U'| = |\nabla u| = 1$ on $F(u)\cap \tilde{A}$, we have $|\Phi'| = 1/|U'| = 1$ on
$\de D_R\cap \{\Re \z = 0\}$, so that $h = \log|\Phi'| = 0$ on $\de D_R\cap \{\Re \z = 0\}$. 
Our goal is to control the oscillation of $\tilde h$ on the imaginary axis.

We claim
that
\begin{equation}\label{eqn:oscillation_linbd}
|\mbox{osc } \tilde h | \le C\e \quad\text{over}\quad \mathcal{F}:=\{\Re \z = 0, 4d \leq |\z| \leq r_{N-2}\}.
\end{equation}
This bound on oscillation of the conjugate $\tilde{h}$ over $\mathcal{F}$ measures the turning of $\nabla u$ over $$\Phi(\mathcal{F}) \supseteq F(u)\cap (B_{1/c}\setminus B_{cd}),$$ hence \eqref{eqn:oscillation_linbd} is the same as the Lipschitz bound we are aiming for.   

Equation \eqref{eqn:oscillation_linbd} is proved by using the bound $|h| \le C\e$ on $D_R$ to
estimate the integral of $|\nabla \tilde h| = |\nabla h|$ on the semicircle
$|\z| = r_{N-2}$, $\Re \z >0$ and the part of the imaginary axis $\Re \z = 0$, $4d \le |\Im \z| \le r_{N-2}$. 
The estimate on the semicircle is a routine interior regularity estimate.  On the imaginary
axis, we have $h=0$, and hence
\[
|\nabla \tilde h(iy) | = |\nabla h(iy) | = |(\de/\de x) h(iy)| \quad (\z  = x+iy).
\]
Integrating in $y$, the oscillation of $\tilde h$ is dominated by the integral of the absolute value of the
flux of $\nabla h$ through the portion $4d \le |\Im \z| \le r_{N-2}$ of the imaginary axis.   We control the flux
of $\nabla h$ using a barrier function (majorant of $h$). 

\begin{lemma} \label{lem:semicircle} If $b$ is the harmonic function in the half annulus $\d < |\z| < 1$, $\Re \z>0$,
with boundary values $b(e^{i\theta}) = b(\d e^{i\theta}) = 1$ for $0 < \theta < \pi$, and $b(iy) = 0$, $\d < y < 1$,
then there is an absolute constant $C$ such that 
\[
\int_{2\d}^{1/2} |(\de/\de x) b(iy)| \, dy \le C.
\]
\end{lemma}
\noi
This lemma is proved in  \cite[Lemma 9.3]{JK} as follows.  Interior regularity yields the bounds
$b(e^{i\theta}/2) \le C \sin \theta$ and $b(2\d e^{i\theta}) \le C\sin \theta$.
Then one finds an explicit barrier in the smaller half annulus by conformal mapping and separation of variables.

The crucial feature of the linear estimate for the flux in Lemma \ref{lem:semicircle} is that 
it is independent of $\d>0$ and yields control on the oscillation down to scale $2\d$ comparable to $\d$. 
Rescaled from a unit sized half annulus to the set $4d < |\z| < r_{N-2}$, $\Re \z >0$
we obtain estimate \eqref{eqn:oscillation_linbd} by using a suitably scaled multiple of
$b$ to majorize $h$.  This concludes the proof of Theorem \ref{thm:effremov-weak}.
\end{proof}

\begin{proof}[Conclusion of the proof of Theorem \ref{thm:removable}]
Let $u$ be as in the statement of the theorem. Let $\eps_0$ be the absolute constant of Theorem \ref{thm:effremov-weak} and fix $0<\eps<\eps_0$. According to Lemma \ref{lem:flatness}, there exist positive constants $\d_0=\d_0(\eps)$ and $C=C(\eps)$ such that if $u$ is a solution satisfying the topological assumption (A) in $B_1\setminus B_{\d}$ for $0<\d <\d_0$, then for every $r\in (C\d, 1/C]$
\begin{equation}\label{thm:removable:eqn1}
    |u \chrc{T}(\rho x) - P(x)| \leq \eps r \quad \text{in} \quad B_{r}\setminus B_{C\d}, \quad \text{for some rotation }  \rho = \rho_r
\end{equation}
where $T$ is the component of the positive phase of $u$, whose closure intersects $\de B_{\d}$.
Define $$v(x) := Cu\chrc{T}(x/C) \quad \text{for}\quad x\in B_{1}\setminus B_{C^2\d}.$$
Then, for $d=C^2\d$, $v$ satisfies the topological assumption (A) in $B_{1}\setminus B_{d}$, and because of \eqref{thm:removable:eqn1}, $v$ is $\eps$-flat on all scales $r\in (d, 1)$. We can thus apply Theorem \ref{thm:effremov-weak} to $v$ and complete the proof.
\end{proof}

\section{The double hairpin solution} 

In this section we will provide an alternative description of the family of double hairpin solutions 
$$
\{H_a: \R^2 \to \R: H_a(x) = a H(x/a), ~ a>0\},
$$
discovered by Hauswirth, H\'elein and Pacard \cite{HHP}, which we will use in the proof of the rigidity of the approximate double hairpin solutions. 
Recall that the double hairpin solutions have simply connected positive phase
\[
\Omega_a:= \{H_a>0\} = a \Omega_1
\]
bounded by two catenary curves of free boundary 
\[F(H_a)=\{(x_1, x_2): |x_2/a| = \pi/2 + \cosh (x_1/a)\},\]
and that $H_a$ is invariant under reflection symmetry with respect to both the $x_1$ and the $x_2$-axis:
\[
H_a(\pm x_1, x_2) = H_a(x_1, x_2) = H_a(x_1, \pm x_2).
\]

\noi Since the part of the positive phase of $H_a$ contained in the \emph{right} half-plane \mbox{$\mathcal{D}_a=\Omega_a\cap \{x_1>0\}$} is also simply connected, there one can define a holomorphic extension of $H_a$
\[
U_a:= H_a + i \tilde {H}_a: \mathcal{D}_a \to \C,
\]
where $\tilde {H}_a$ is the harmonic conjugate of $H_a$ in $\mathcal{D}_a$, chosen in a way that $\lim_{z\to 0}\tilde{H_a}(z) = 0$. 
\begin{figure}[h]
\centering 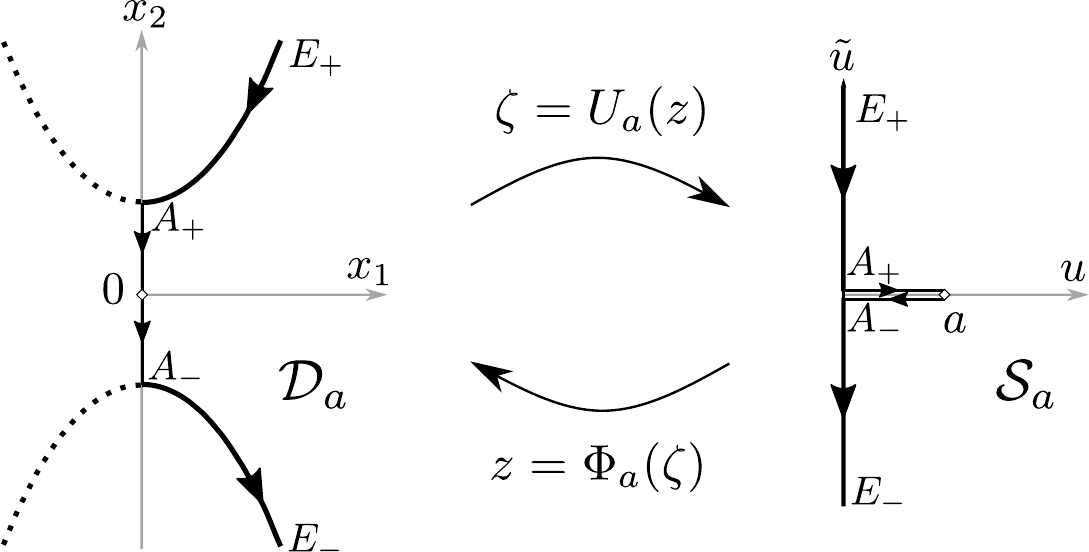 \caption{The conformal diffeomorphism $U_a = H_a + i \tilde{H}_a$ mapping the right half of the positive phase $\mathcal{D}_a = \O_a \cap \{x_1>0\}$ onto the slit domain $\mathcal{S}_a$.}
\label{HPConformal}
\end{figure}
Then  $U_a$ maps  $\mathcal{D}_a$ biholomorphically onto  $\mathcal{S}_a$, the right half-plane minus a horizontal slit of length $a$ (see Figure \ref{HPConformal}):
\[
U_a(\mathcal{D}_a) = \mathcal{S}_a:=\{\zeta = u+i\tilde{u}\in \C: u>0\}\setminus \{\zeta = u: 0< u \leq  a\},
\]
and has an inverse 
\[
\Phi_a:= (U_a)^{-1}: \mathcal{S}_a \to \mathcal{D}_a
\]
whose derivative $d\Phi_a/d\z = e^{\varphi_a(\z)},$ where 
\[
\varphi_a(\z) = -\frac{1}{2}\log (\z - a) + \frac{1}{2}\log (\z+a).
\]
In the formula above we use the standard branch of the logarithm. Note that $a=H_a(0)$ is precisely the value of $H_a$ at the unique \emph{saddle point} at the origin. Observe also that the real part $\Re \varphi_a(\z) \to 0$ as $\z \to i\tilde{u}$  in accordance with the free boundary condition  $\lim_{z \to F(H_a)}|U_a'(z)|=  \lim_{z \to F(H_a)}|\nabla H_a(z)|=1$: 
$$
\lim_{\z \to i\tilde{u}}|e^{\varphi_a (\z)}|=\lim_{\z \to i\tilde{u}}|\Phi_a'(\z)|= \lim_{z \to F(H_a)}\frac{1}{|U_a'(z)|}=1. 
$$
Integrating the expression for $d\Phi_a/d\z$, one can write down an explicit formula for $\Phi_a$:
\[
\Phi_a(\z) = a((\z/a)^2-1)^{1/2} + a \log\left(\z/a + ((\z/a)^2-1)^{1/2}\right) = a \Phi_1(\z/a).
\]
Now the double hairpin solution $H_a$, \emph{restricted to} the right half-plane $\{x_1>0\}$, is given by
\[
H_a(z) = \left\{ \begin{array}{ll} \Re \z & \text{when}\quad z = \Phi_a(\z), \quad \z \in \mathcal{S}_a\\ 0 & \text{for all other } z \in  \{\Re z >0\}. \end{array} \right.
\]

Observe that in the blow-down limit $a\to 0$, $H_a(x)$ converges to the \emph{wedge solution} $W_1(x)= |x_1|$.  A slightly different kind of blow-down limit arises when we consider appropriate vertical translations of $H$ in addition to rescaling. It is not hard to show that 
\[
\frac{a}{t}H(t x_1/a, t x_2/a + \pi/2 + \cosh t) \to TP_{2a}(x_1-a,x_2) \quad \text{as}\quad t\to \infty 
\]
where $TP_{2a}(x)$ is the \emph{two-plane} solution of \eqref{FBP}
\[
TP_{2a}(x) = (x_1)^+ + (-x_1 - 2a)^+, \qquad a>0.
\]

\section{Rigidity of solutions with simply connected positive phase}

In this section we sketch the proof of Theorem \ref{thm:mainsimplyconn} 
(Theorem 1.3 \cite{JK}).   Consider a classical solution $u$ in $\DD$.  If
only one component of $\DD^0(u)$ meets $B_{2c}(0)$, then there
is a bound on the curvature of $F(u)$ in $B_c(0)$, depending only on $c$
(see Proposition 8.3 of \cite{JK}).  This corresponds to part (a) of
our theorem.  

If two components of $\DD^0(u)$ meet
$B_{2c}(0)$, and $c$ is sufficiently small, then Proposition 8.4 of \cite{JK}
says that at a larger scale, the configuration is approximated
by the wedge solution $|x_1|$ and can be described as in Lemma \ref{lem:flattypes}
(c).  For any $\e>0$ and $M<\infty$ we can choose $c$ sufficiently small and $r$
satisfying $Mc< r < 1/2$,  such that 
\[
|u(x) - |x_1|| \le \e r, \quad x\in B_r.
\]
Moreover, $\DD^0(u) \cap B_r$ consists of two components in
the strip $|x_1| < \e r$, both of which intersect $B_{2c}$.   
What remains is to estimate how closely this configuration resembles 
a double hairpin.

To begin with, we treat the scale comparable to the distance
separating the two components of the zero phase.  Corollary 10.4 of \cite{JK}
implies that for any $\e>0$ and any $M<\infty$, we can choose $c$ 
sufficiently small that $u$ has exactly one critical point $z_0$ in $B_{2c}(0)$
(a saddle point) and, defining $a = u(z_0)$, we have 
\begin{equation} \label{eq:localhairpin}
|\nabla^j [u(z_0+x) - H_a(\rho x)]| \le \e a^{1-j}, \quad |x| \le Ma, \quad 
j = 0, \, 1, \, 2, \,  3.
\end{equation}
for some rotation $\rho$.   In particular, the distance between
the two components of the zero phase is very close to $(2 + \pi)a$.
Recall that the distance between the components is at most $4c$. 

We summarize the proof of \eqref{eq:localhairpin} as follows.
Using a compactness argument based on Theorem \ref{thm:mainglobal},
one finds that given any $\e'>0$ and any $M'<\infty$, there is $c$ sufficiently small  such that for any $u$ with two zero components in $B_{2c}(0)$ 
there is $a'< c$ such that 
\[
|u(z'+ x) - H_{a'}(\rho \, x)| \le \e' a', \quad x\in B_{M'a'}(0)
\]
for some $z'\in B_{2c}(0)$ and some rotation $\rho$. The choice
of $z'$ and $a'$ are not unique.  They can be changed slightly without
violating the inequality.   To make a more precise choice,
following the proof of the corollary, 
we note that this uniform estimate,
non-degeneracy of $u$ and standard Alt-Caffarelli free boundary regularity estimates 
imply  that on $|x| \le M'a'/2$ the solution  and the double hairpin are close
in $C^3$ norm.   Because $H_{a'}$ has a unique
nondegenerate critical point at the origin, $u$ has exactly one critical point
$z_0$ very close to $z'$.  Hence $a = u(z_0)$ is very close to $a'$,
and \eqref{eq:localhairpin} follows.

To conclude the proof, we construct a conformal mapping
with a neighborhood of the double hairpin of size
comparable to $c$ rather than size $a$, which, in the case
not covered by the estimate \eqref{eq:localhairpin}, is much smaller than $c$.
Recall that the conformal mapping $U_a$ in Figure 2 maps half of
the double hairpin positive phase, $\Dee_a = \O_a \cap \{x_1>0\}$ to 
the slit half plane $\S_a$.  

The function $u$ is harmonic on $B_{10c}(z_0) \cap \DD^+(u)$.  Because that
set is simply connected, the harmonic conjugate $\tilde u$ is defined
up to an additive constant and we can set $U = u + i\tilde u$. 
By analogy with the model case of the double hairpin, denote by $\mathcal{A}\_\pm$
the curves of steepest descent of $u$ from the critical point $z_0$ to
the two components of the zero set of $u$.  The union of these two
curves divides $B_{10c}(z_0)\cap \DD^+(u)$ in half.  Fix one such
half and call it $\Dee$.  Orient the boundary of $\Dee$ so the
region is to the left.  In counterclockwise order, the boundary of $\Dee$ 
follows an arc of $\de B_{10c}(z_0)$, then an arc of $F(u)$, which we label
$\mathcal{E}_+$, then the arc from the component of $\DD^0(u)$
to the critical point, which we label $\mathcal{A}_+$, the arc from
the critical point to the other component of $\DD(u)$ which
we label $\mathcal{A}_-$, then a second arc of $F(u)$ which we label
$\mathcal{E}_-$, returning to the arc of $\de B_{10c}(z_0)$.   Recall that 
$u$ is zero on $F(u)$ and $u(z_0) = a$.  With the normalization for the 
conjugate $\tilde u$ given
by $\tilde u(z_0) = 0$, and noting that $\tilde u$ is constant on the steepest descent lines, we find that $U$ goes from $a$ to $0$ along both $\mathcal{A}_\pm$.
Thus $U$ is a conformal mapping from $\Dee$ onto its image
in the slit right half plane $\S_a$.
Moreover, referring to the diagram in Figure 2 of
$\S_a$ bounded by $E_\pm \cup A_\pm$, under the mapping
$U$, $\mathcal{E}_+$ maps to a subset
of $E_+$, $\mathcal{A}_+$ maps to $A_+$, $\mathcal{A}_-$ 
maps to $A_-$, $\mathcal{E}_-$ maps to a subset of $E_-$.

On $\Dee$, $U$ has a well-defined inverse and we set
\[
\Psi : = U^{-1} \circ U_a,
\]
a mapping that sends the subset $\Phi_a(U(\Dee))$ of $\Dee_a$ to $\Dee$.   
By the same reasoning, $U$ sends
$\Dee' : = (B_{10c}(z_0) \cap \DD^+(u)) \setminus \Dee$ to $\S_a$.
When the boundary of $\Dee'$ is traced with $\Dee'$ to the left, i.~e.,
counterclockwise, the boundary correspondence at the overlap
$\mathcal{A}_\pm$ is traced in the opposite orientation. 
$\Psi : = U^{-1}\circ U_a$ is also well defined on the complementary subset
$-\Phi_a(U(\Dee'))$ of $\O_a \setminus \Dee_a$.  
The two parts of the mapping 
$\Psi$ are consistent, that is, continuous across $\mathcal{A}_\pm$
because of the rotation by $180^\circ$.   Put another
way, both $U_a$ and $U$ are double coverings of $\S_a$.
Thus the composition $U^{-1}\circ U_a$ is well defined 
because the branch points coincide at the point $a$, the
tip of the slit in $\S_a$.

It remains to show that the estimates stated in Theorem \ref{thm:mainsimplyconn}
are valid for $\Psi$.    There is nothing to prove if $c \le Ma$, so we can assume $a \ll c$. 
As in the proof of the effective removable singularities
theorem, we consider the logarithmic derivative, 
$\Psi' = e^{f + i\tilde f}$, defined on, say, $B_{8c}\cap \O_a$, which
is a subset of $\Phi_a\circ U(\Dee) \cup (-\Phi_a\circ U(\Dee'))$.  
The estimates we want to prove for $\Psi$ are then deduced from
estimates for a linear boundary value problem for the harmonic
function $f$ as follows.     

On the annulus $B_{8c}(z_0)\setminus B_{4c}(z_0)$, 
$F(u)$ as close as we like  (in say $C^3$ norm) to four horizontal lines near
the axes.  Thus for any fixed $\e>0$ we can choose $c$ sufficiently
small and $M$ sufficiently large that $\Psi$ is within $\e$ of an isometry on
$\O_a \cap B_{8c}\setminus B_{4c}$.
\[
|\Psi'| = 1 + O(\e) \quad \text{on} \quad \O_a \cap (B_{8c}\setminus B_{4c}).
\]
Combining this with the maximum principle for $f$, and $e^f = |\Psi'| = 1$ on $B_{8c} \cap \de\O_a$, we 
have 
\[
|f|  \le  \e \ \mbox{on} \ B_{8c} \cap \O_a; \quad f = 0 \ \mbox{on} \ B_{8c} \cap\de \O_a \, .
\]
From this we can deduce an estimate up to boundary the gradient of $f$:
\[
|\nabla f| \le C\e \ \mbox{on} \ B_{4c} \cap \O_a,
\]
with an absolute constant  (in particular one that is independent of $a$).
This is because $H_a$ is a positive harmonic function on $\O_a$ with zero boundary conditions
and satisfies $|\nabla H_a| = 1$ on $\de \O_a$, and by standard barrier argument
this harmonic function majorizes Green's function, giving a slope bound that is 
independent of $a$.   This bound on $|\nabla f|$ gives the bounds on $\Psi'$ and $\Psi''$
of Theorem \ref{thm:mainsimplyconn}.  (See \cite{JK} for further details.)

\begin{remark} One special feature of the proofs of Theorems \ref{thm:removable} 
and \ref{thm:mainsimplyconn}  that limit 
them to two dimensions is that they rely on the classification of global solutions, which
is only carried out in two dimensions.  The other is the reliance on conformal mapping.   
In the proof of Theorem \ref{thm:removable}, the nonlinear estimates required are reduced via the conformal
mapping $\Phi$ to estimates for $h = \log|\Phi'|$ in a {\bf linear} boundary value problem on half annuli of
the form $\{z= x_1 + ix_2\in \C: x_1 >0, \ cd < |z| < 1/c\}$.  See the proof of Theorem \ref{thm:effremov-weak},
and, in particular, the uniform flux estimate Lemma \ref{lem:semicircle}.  The key point is the uniformity in $d>0$ 
as $d\to 0$.    Likewise in the proof of Theorem \ref{thm:mainsimplyconn}, estimates for $\Psi'$ and $\Psi''$ are reduced
to linear estimates for $f = \log |\Psi'|$ on $B_{8c}\cap \O_a$, and the key point is a different flux estimate, stated
in the paragraph just preceding this remark, that is uniform
in the parameter $a>0$ as $a \to 0$. 
\end{remark}

\section{Entire multiply connected solutions}

After discovering the double hairpin solution in \cite{HHP}, the authors conjectured that the list of entire classical solutions having \emph{connected, finitely connected} positive phase, is very short, consisting, up to rigid motion, of
just $P$,  $\{H_a\}_{a>0}$ (with simply connected positive phase), and the disk complement solution 
$L_R(x):=R(\log(x/R))^+$, $R>0$.

In the special case of simple connectivity, the conjecture was confirmed by Khavinson, Lundberg and Teodorescu \cite{KhavLundTeo}. Almost concurrently, Martin Traizet \cite{Traizet} resolved the full conjecture by drawing out a remarkable correspondence between classical solutions $u$ with $|\nabla u|<1$ in their positive phase \mbox{$\O \subseteq \R^2$} and \emph{minimal bigraphs} --- complete embedded minimal surfaces in $\R^3$ having a plane of symmetry, over which each half is a graph. Specifically, the map 
\begin{equation*}\label{Traizet}
\begin{array}{lrcc}
T :& \O  & \rightarrow &  \R^2 \times  \R^+\\
& z & \rightarrow & (X_1(z), X_2(z), u(z)),
\end{array}
\end{equation*}
where $X_1$ and $X_2$ are obtained by integrating the differential 
\[
 dX_1 + idX_2 = \frac{1}{2}\Big(d\bar{z} -  \left(2\frac{\de u}{\de z}\right)^2 dz\Big),
\]
can be shown to define a minimal immersion of the positive phase $\O = \{u>0\}$ of a classical solution $u$ of \eqref{FBP} into the upper half-space $\R^2 \times  \R^+$. The free boundary condition  $|\nabla u| =1$ on $F(u)$ means that the immersed mimimal surface $T(\O)$ attaches \emph{orthogonally} to the plane $X_3 = 0$, so that $T(\O)$ can be completed  to a minimal surface $M$ that has a plane of symmetry. Furthermore, Traizet showed that if $|\nabla u|<1$ in $\O$, then $M$ is an \emph{embedded} minimal surface, and if $u$ is an \emph{entire} classical solution satisfying $|\nabla u|<1$ in $\O$, then $M$ is a \emph{complete} embedded minimal bigraph. The converse is also true: any complete embedded minimal bigraph $M$ gives rise to an entire classical solution of \eqref{FBP} with $|\nabla u| <1$ in the positive phase. 

Under the correspondence, an entire classical solution with finitely connected positive phase (that is not the one plane $P$) gets associated to a minimal bigraph $M$ with two ends and finite topology. By the classification theorem of Schoen \cite{schoen1983uniqueness} $M$ has to be a catenoid. In this way, the disk-complement solution $L_R$ is revealed to be the counterpart to the vertical catenoid, and the double hairpin $H_a$ --- to the horizontal catenoid!

Utilizing the correspondence further, Traizet rediscovered a family of \emph{infinitely} connected periodic solutions of \eqref{FBP} that correspond to the \emph{Scherk simply periodic surface}, viewed as a minimal bigraph.  To celebrate this connection, we will be referring to these solutions as \emph{Scherk solutions}. In fact, the family had been first discovered by Baker, Saffman and Sheffield \cite{BakerHollowVortices} in the fluid dynamics literature, and below we will provide a description using conformal mapping more in tune with their approach.   

The family of Scherk solutions comprises, up to rigid motions,
\begin{equation*}\label{Scherk}
\{ S_{s,a}: \R^2 \rightarrow \R: S_{s,a}(x) = a S_{s}(x/a):  0<s<1, 0<a<\infty \},
\end{equation*}
where $S_s(x_1,x_2)$ is $2\pi $-periodic along $x_2$, reflection-symmetric with respect to both the $x_1$ and $x_2$ coordinate axes
\[
S_s(x_1, x_2+2\pi) = S_s(x_1, x_2), \quad S_s(\pm x_1, x_2) = S_s(x_1, x_2) = S_s(x_1,\pm x_2),
\]
and blows down to the \emph{wedge solution} $W_s$: 
$$\lim_{a\to 0} S_{s,a}(x) = W_s(x) = s|x_1| \quad \text{of slope} \quad 0<s<1. $$
Denote the positive phase of $S_s$ by
\[
\Omega^{\text{BSS}}_s := \{S_s > 0\}.
\]
and let the perimeter of each closed loop $\g$ of $F(S_s)$ be $2l$ (see Figure \ref{Scherk_slopes} for a plot of $F(S_s)$). 
\begin{figure}[h]
\centering \includegraphics[scale=0.6]{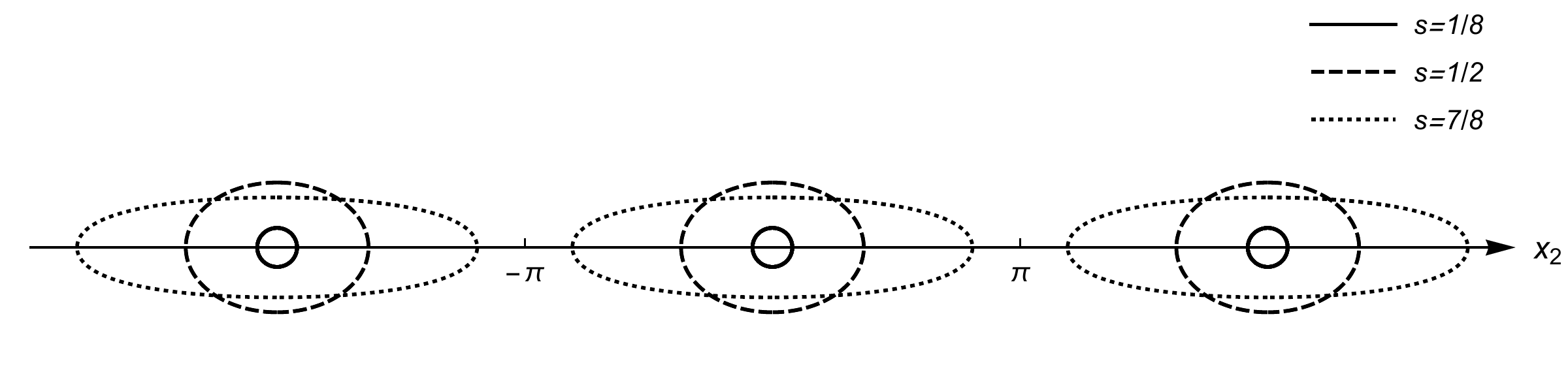} \caption{Mathematica plot of the free boundary of the Scherk solution $S_s(x_1,x_2)$ for asymptotic slopes $s=1/8$, $s=1/2$ and  $s=7/8$. Note that in the diagram $x_2$ is the horizontal axis.}
\label{Scherk_slopes}
\end{figure}

\noi Over the domain $$\mathcal{D} = \{(x_1, x_2)\in \R^2: x_1 > 0, |x_2| < \pi \}$$ the part of the positive phase $\mathcal{D}^{\text{BSS}}_s:=\Omega^{\text{BSS}}_s \cap \mathcal{D}$ is simply connected and bounded by $\partial \mathcal{D}$ and the \emph{right} half of $\g$ (see Figure \ref{Scherk_conf}). There one can define a holomorphic extension of $S_s$ 
\[
U_s^{\text{BSS}} := S_s + i \tilde{S_s} : \mathcal{D}^{\text{BSS}}_s \to \C
\]
normalized so that $U_s^{\text{BSS}}$ maps $\mathcal{D}^{\text{BSS}}_s \cap \{x_2 = 0\}$ into the real axis. 
\begin{figure}[h]
\centering 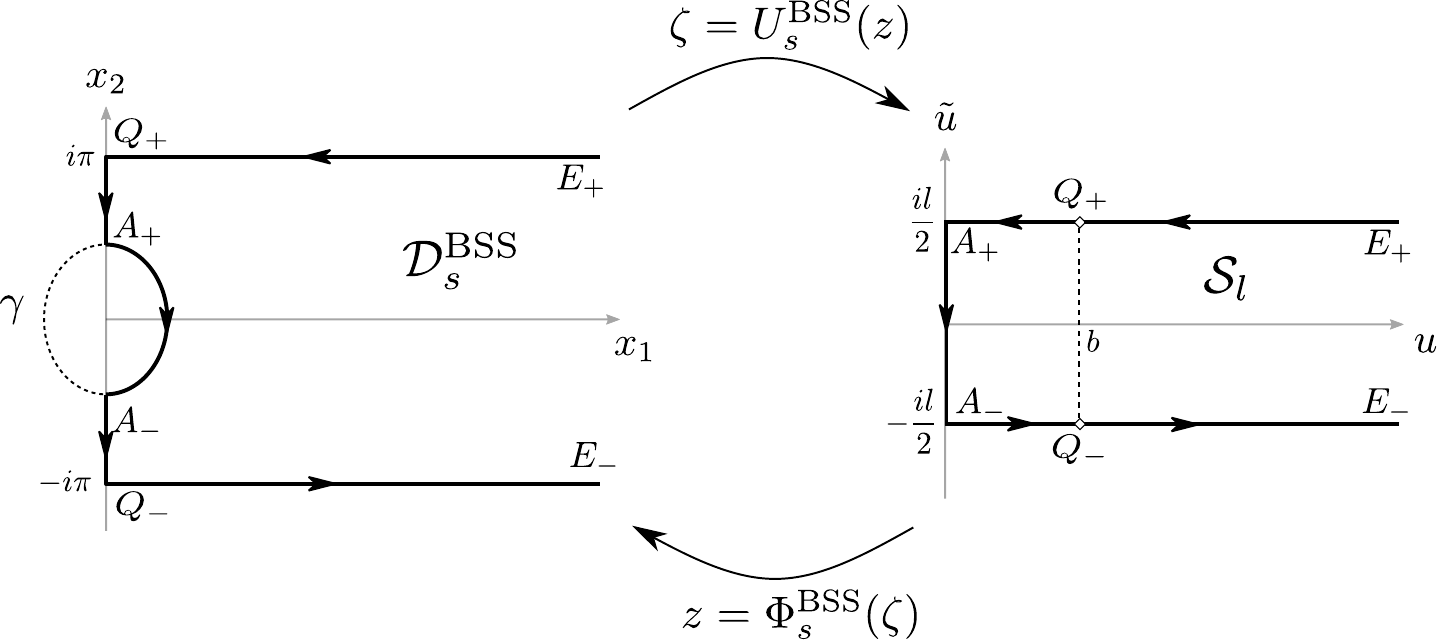 \caption{Mapping the subdomain $\mathcal{D}^{\text{BSS}}_s$ of the positive phase of the Scherk solution $S_s$ conformally onto the strip $\mathcal{S}_l$ under $U_s^{\text{BSS}} = S_s + i \tilde{S}_s.$ Note that $Q_{\pm}$ is a saddle point of $S_s$ with $Q_{\pm}A_{\pm}$ and $Q_{\pm}E_{\pm}$ being a steepest descent and a steepest ascent path from $Q_{\pm}$, respectively. }
\label{Scherk_conf}
\end{figure}

For the sake of notational simplicity we shall drop the superscript ``BSS.'' The map $U_s$  is biholomorphic onto its image $$U_s(\mathcal{D}_s)  = \mathcal{S}_l:=\{\zeta = u+i \tilde{u}\in \C: u > 0, |\tilde{u}| < l/2 \} $$ and has a holomorphic inverse
\[
\Phi_s = (U_s)^{-1}: \mathcal{S}_l\rightarrow \mathcal{D}_s
\]
given by $\displaystyle d \Phi_s/d \zeta = e^{\varphi_s(\z)}$, where 
\begin{equation}\label{eqphi}
\varphi_s(\z) = -\frac{1}{2} \log\left(e^{2\pi(\z-b)/l} + 1 \right) + \frac{1}{2} \log\left(e^{-2\pi(\z+b)/l} +1 \right)  + \frac{\pi \z}{l}.
\end{equation}
In the formula above we use the standard branch of the logarithm. Observe that $b=S_s(Q_{\pm})$ is precisely the value of $S_s$ at the \emph{saddle points} $Q_{\pm}= \pm i\pi$ between two consecutive loops (the half-logarithmic singularities of $\varphi_s$ as $\zeta \to b\pm i l/2$ precisely account for the saddle point behaviour of $S_s$ near $Q_{\pm}$). Note also that $\Re \varphi_s(\z) \to 0$ as $\z \to i\tilde{u}$ in accordance with the free boundary condition. 

The relation between $b$, $s$ and $l$ can be found by observing
\begin{equation}\label{eqn:Scherk_ParamRel}
\frac{1}{s} = \lim_{\Re z \to \infty} \frac{1}{U_s'(z)} =  \lim_{\Re \z \to \infty}\Phi_s'(\z) =  \lim_{\Re \z \to \infty} e^{\varphi_s(\z)} = e^{\pi b/l}.
\end{equation}
In addition, a simple flux computation for $\nabla S_s$ through $\partial (\mathcal{D}_s \cap\{x_2\leq T\})$ in the limit  $T \to \infty$ yields
\begin{equation}\label{eqn:Scherktrivia1}
l = 2\pi s
\end{equation}
so that \eqref{eqn:Scherk_ParamRel} implies that in terms of the asymptotic slope $s$,
\begin{equation}\label{eqn:Scherktrivia2}
b =  2 s \log(1/s).
\end{equation}
From \eqref{eqphi} we can see that $\Re\varphi_s(u + i\tilde{u})$ is even in $\tilde{u}$ and $\Im \varphi_s(u + i\tilde{u})$ is odd in $\tilde{u}$, so that 
$\Im \int_0^{u+i\tilde{u}} e^{\varphi_s(\eta)} d\eta$ defines an odd function in $\tilde{u}$. Hence,  
\[
\Phi_s (\z)= \int_0^{\z} e^{\varphi_s(\eta)} d\eta + c_s
\]
where $c_s$ is an appropriate \emph{real} constant that ensures $\Re \Phi_s(\zeta) \to 0$ as $\zeta \to \pm il/2.$  Now the domain $\mathcal{D}_s = \Phi_s(\mathcal{S}_l)$ and the Scherk solution $S_s$ over $\mathcal{D}$ is given by
\[
S_s(z) := \left\{ \begin{array}{ll} \Re \z & \text{when}\quad z = \Phi_s(\z), \quad \z \in \mathcal{S}_l\\ 0 & \text{for all other } z\in \mathcal{D}.  \end{array} \right.
\]
One can also compute an equation for the loop $\g$ in the $x_1, x_2$ coordinates. We know that the right half of the loop $\g$ in the $x_1x_2$--plane is the image of $\{\z=i\tilde{u}: |\tilde{u}|<l/2\}$ under $\Phi_s$, where 
\[
\Phi_s (i\tilde{u})= i \int_0^{\tilde{u}} e^{\varphi_s(i t)} dt + c_s =  \int_0^{\tilde{u}} \left(- \sin\big(\Im \varphi_s(it)\big) + i \cos\big(\Im \varphi_s(it)\big)\right) dt + c_s
\]
with 
\[
\Im \varphi_s(it) = \frac{\pi}{l} t - \arctan\left(\frac{e^{-2\pi b/l} \sin (2\pi t/l) }{1+ e^{-2\pi b/l}\cos (2\pi t/l)}\right) = \frac{t}{2s} -\arctan\left(\frac{s^2 \sin (t/s) }{1+ s^2\cos (t/s)}\right).
\]
Then the Euclidean coordinates of $\gamma$ integrate to the parametric equations 
\begin{align*}
x_1(\tilde{u}) &= (1 - s^2) \log \left(\frac{\sqrt{1+s^4 + 2s^2 \cos(\tilde{u}/s)} + 2s \cos(\tilde{u}/(2s))}{1-s^2}\right) \\
x_2(\tilde{u}) & =  (1 + s^2) \arctan \left(\frac{2s \sin(\tilde{u}/(2s))}{\sqrt{1+s^4 + 2s^2 \cos(\tilde{u}/s)}}\right)
\end{align*}
from which $\tilde{u}$ can be eliminated to yield
\begin{equation}\label{eqn:ScherkBubble}
(1-s^2) \cosh\left(\frac{x_1}{1-s^2}\right) = (1+s^2) \cos\left(\frac{x_2}{1+s^2}\right).
\end{equation}

\section{Final Remarks and Open Questions}

In this section we formulate some questions about which solutions can arise as limits of classical solutions 
in dimensions two and three as well as analogous questions for solutions to semilinear elliptic equations.

Traizet has shown that the only entire classical solutions with finite topology
are the half plane, the double hairpins, and the circle.   In keeping with this characterization,
we conjecture that these two configurations are in essence the only ones one sees
at finite scale. 
\begin{conjecture}  \label{conj:finitetop} 
There are absolute constants $C$ and $r>0$ such that 
for every classical solution $u$ in the unit disk $\DD$ such
that the free boundary $F(u)$ consists of at most N curves,
there are points $p_j\in B_{2/3}(0)$, $j = 1,\, 2,\, \dots \, M$, $M\le N$
such that $B_r(p_j)\cap F(u)$ is nearly isometric to a double hairpin or to a single circle, 
and in the complement
\[
B_{1/2}(0) \setminus \bigcup_{j=1}^M B_r(p_j)
\]
$F(u)$ has  curvature bounded by $C$.
\end{conjecture}
If the radius of an approximate circle in $F(u)$ is $\rho$,  then a compactness argument and the theorem of
Traizet shows that are no other circles within a multiple of $\rho$ of this circle.  The 
conjecture asserts that there are no other curves within $r$ of this circle, which is a much
larger distance than $\rho$.  This is in the same spirit as our theorem about double hairpins
in the simple connected case.

Recall that using  Caffarelli's notion of viscosity solutions we can show that 
blow down limits of entire classical solutions in the plane are of the form
$s_+x_1^+ + s_-x_1^-$ with $0 \le s_- \le s_+ \le 1$.  If both $s_\pm >0$,   
then we showed, using the fact that the solution is also a variational solution, 
that the slopes are equal:
$s=s_+ = s_->0$.  Note further that in the simply connected case, we also show that $s=1$, and 
that this arises in the case of the blow down of the simply connected double hairpin solution. 
The blow down limit of multiply connected Scherk type examples yields wedges of the form 
$W_s(x) = s|x_1|$ for each $s$, $0 < s <1$.     So far, at least, there does not appear to
be a situation in which unequal slopes $0< s_- < s_+$ arise.

We refer to the process in which a free boundary with large curvature blows
down to a straight line as a \emph{singular limit} and ask what singular limits can arise on
bounded domains.  We distinguish the case of finite topology from the 
case in which the connectivity tends to infinity.

\begin{question} \label{quest:finitetop}  
Suppose $u_k$ is a sequence of classical
solutions in the unit disk $\DD$, and $F(u_k)$ is an approximate
double hairpin that tends to a single connected curve.  Can one characterize
the limit curve in $B_{1/2}(0)$?  What pairs of connected limit curves can arise as limits of two double hairpins?  
More generally, what are all singular limits of $F(u_k)$ in the case of finite topology  
($F(u)$ has at most $N$ curves in $\DD$)? 
\end{question}
We expect limits of the classical solutions $u_k$ in Question \ref{quest:finitetop} to solve
a two-sided boundary problem with slope $1$ on each side, in 
keeping with the example of the entire hairpin, which tends to the wedge solution $W_1$.

\begin{question} \label{quest:infinitetop}  
Suppose $u_k$ is a sequence of classical
solutions in the unit disk $\DD$, the number of components of $F(u_k)$ tends infinity 
but the limit in Hausdorff distance is a connected curve.  Can one characterize
the limit curve in $B_{1/2}(0)$?  
\end{question}
We expect limits of the classical solutions $u_k$ in Question \ref{quest:infinitetop} to solve
a two-sided boundary problem with equal slope on each side, but
not necessarily slope $1$, in keeping with the Scherk example, which tends to a wedge
solution $W_s$. 

Just as every minimal surface minimizes area on sufficiently small balls, any classical solution
to the free boundary problem minimizes the functional in sufficiently small disks with fixed
boundary conditions.  On the other hand, the hairpin and the circle are not stable on larger disks.
\begin{question} \label{quest:Morse}  Is there a relationship between the topology
of solutions and their Morse index relative to the Alt-Caffarelli functional at appropriate scales?
\end{question}
\noi Results of this flavor abound in the classical minimal surface literature --- see \cite{chodosh2016topology} and references therein. Similar questions were studied by \cite{KWangWei-AC-finite-ends, KWang-MorseIndex-FB}  in the case of the Allen-Cahn equation and free boundary problem variants of it in 2 dimensions.
A construction of a solution with higher Morse index in a two-phase free boundary problem can be found in \cite{JP}.

As we have seen, the generalized notion of viscosity solutions introduced by Caffarelli is 
crucial to the study of higher order critical points of the functional.  But one of
the original purposes was to study solutions from a particular
version of the Perron process in which one considers infima of strict supersolutions.
Caffarelli in \cite{CafIII} succeeded in proving a partial regularity result for such solutions
that shows that their free boundaries are smooth except on a set of zero $(n-1)$-Hausdorff
measure.  Many years ago, when one of us asked Luis what the impediment to further regularity was,
he replied that he could not rule out a highly disconnected free boundary formed
from a collection of smaller and smaller bubbles.  We see this phenomenon 
vividly in the Scherk example.   But if the entire solutions we already know
about are the only ones, then this suggests that the Perron solutions are more like
minimizers of the functional.   At sufficiently large scale, neither the double hairpin solution nor the Scherk type 
solution is the infimum of strict supersolutions.    (The strictness is crucial:  solutions are themselves supersolutions,
so they can always be realized as the infimum of supersolutions.)

\begin{conjecture} \label{conj:Perron}   
Solutions obtained as the infimum of strict supersolutions
as in \cite{CafIII} are classical solutions in dimension two.  
\end{conjecture}

\begin{question} \label{quest:Perron}  Are solutions obtained as the infimum of strict supersolutions
as in \cite{CafIII} also stable for the Alt-Caffarelli functional? or well approximated by stable solutions,
hence as regular as such solutions?
\end{question}

We turn now to higher dimensional free boundary problems.   One of the first problems is to
classify entire solutions with given topological constraints.  The most basic question is the following.

\begin{question}  Suppose that $u$ is an entire classical solution in $\R^3$ to the one-phase free boundary 
problem whose positive and zero phase are both contractible.  Is $u$ equal to a half space solution $x_1^+$
after rigid motion?  If the solution is defined in the unit ball $\B$ and both
the positive and zero phase are contractible, is its free boundary
smooth with uniform bounds on the concentric ball of radius $c$ for some
absolute constant $c>0$? 
\end{question}

So far, the only examples of entire classical solutions in three dimensions besides the half space solution 
are the product domain of the disk complement with a real line, the product domain of the double hairpin of Hauswirth {\em et al.}  with a real line, and a solution
that is symmetric around an axis of rotation constructed by Liu, Wang and Wei \cite{LWW}.  The latter two
solutions blow down to wedge solutions, $|x_1|$.  
As Liu {\em et al.} suggest (see their Remark 2) 
there should also be smooth solutions that blow down to the conic,
axially symmetric  Alt-Cafarelli solution.  More generally, we propose the following questions.

\begin{question} \label{quest:symmetry}  Are there entire homogeneous solutions to the one-phase free boundary
problem that are symmetric with respect to the discrete $\Z_n$ action by rotation around an axis, for each $n$
in analogy with the case of the Alt-Caffarelli example with $\Z_2$ symmetry?  What about other discrete 
subgroups of the rotation group $SO(3)$?
\end{question}

\begin{question} \label{quest:desingularized}  Are there classical entire solutions to the one-phase free boundary
problem that are asymptotic to entire homogeneous solutions at infinity?
\end{question}

There are (at least) two different ways to formulate a conjecture in three dimensions
extending the flat-implies-Lipschitz theorem of Alt and Caffarelli to cases
with holes.  The first way says that if the solution is defined in 
a ball minus a narrow cylinder, each component of the positive phase is trivial, and the free boundary is sufficiently flat, then the only
possibility in addition to the half space solution is a solution that resembles one half of
the product solution associated to the double hairpin. 

To state this conjecture we set $\B_r = \{x\in \R^3: |x| < r\}$, $\CC_r = \{x\in \R^3: x_2^2 + x_3^2 < r^2\}$. 

\begin{conjecture}  
Suppose that $u$ is a classical solution to the one phase free boundary
problem in $\B_1 \setminus \CC_\d$ whose positive and zero
phase are contractible.
There are absolute constants $\d_0>0$, $\e>0$ and $C$ 
such that if $\d\le \d_0$, and the free boundary is flat, i.~e., 
for every $x= (x_1,x_2,x_3)\in \B \setminus \CC_\d$, 
\[
 x_1 > \e \implies u(x) >0 \quad \mbox{and} \ x_1 < -\e \implies u(x) = 0,
\]
then 
\[
F(u) \cap \B_{1/C} \cap \{x_2^2+x_3^2 > (C\d)^2\} = \{(x \in \B_{1/C}: 
x_1  = g(x_2, x_3), \ x_2^2+x_3^2 > (C\d)^2\}
\]
for some function $g:\R^2 \to \R$ with $|\nabla g | \le 1/100$.  
\end{conjecture}
\noi In the case of the product solution associated to the double hairpin removing the cylinder disconnects the positive phase. The conjecture is intended to apply separately to each half of the solution --- with the positive phase in the other half replaced by zero.

The second way to formulate a conjecture about flat solutions is
to consider solutions in a ball minus a small concentric ball. In
that case, the solutions should resemble a half space or one half 
of the solution of Liu {\em et al.} 

\begin{conjecture}  
Suppose that $u$ is a classical solution in $\B_1 \setminus \B_\d$,
whose positive phase and zero phase are contractible. 
There are absolute constants $\d_0>0$, $\e>0$ and $C$ 
such that if $\d\le \d_0$, and the free boundary is flat, i.~e., 
for every $x= (x_1,x_2,x_3)\in \B_1 \setminus \B_\d$, 
\[
 x_3 > \e \implies u(x) >0 \quad \mbox{and} \ x_3 < -\e \implies u(x) = 0,
\]
then 
\[
F(u) \cap \B_{1/C} \cap \{|x| > C\d\} 
= \{(x \in B_{1/C}: x_3  = g(x_1, x_2), \ |x| > C\d\}
\]
for some function $g:\R^2 \to \R$ with $|\nabla g | \le 1/100$.  
\end{conjecture}

Free boundaries and minimal surfaces arise as limits of elliptic semilinear
equations of the form 
\begin{equation}\label{eq:semilinear}
\Delta u = f(u).
\end{equation}
Suppose that $f\in C_0^\infty(\R)$ with the normalization
$\int_{-\infty}^\infty f(u) \, du = a>0$.  
If $u$ is an entire solution to \eqref{eq:semilinear}, and
the limit $\disp U(x) = \lim_{\e \to 0^+} \e \, u(x/\e)$ exists, then, formally, it satisfies
\[
\Delta U = a\delta(U),
\]
with $\delta$ equal to the Dirac delta function.  The correct interpretation
of this formal equation (see \cite{CafSalsa}) is that $\Delta U = 0$ on the sets
$\{U>0\}$ and $\{U<0\}$ and on the interface $U$ satisfies
the overdetermined boundary condition
\[
|\nabla U^+|^2 - |\nabla U^-|^2 = 2a \quad \mbox{on} \ \de \{U>0\}
\]
at least in some generalized sense.  (By $\nabla U^+$ we mean
value of the gradient at the free boundary $\de \{U>0\}$ computed
as a limit from the positive side; $\nabla U^-$ is the limit
from the complement of the positive side.) 
In other words, the blow down limit, if it exists, solves the 
two-phase free boundary problem (or in the case $U\ge 0$ and $U^-\equiv 0$, 
the one 
phase free boundary problem).   The one dimensional solutions
are rigid motions of $U(x) = \a x_1^+ - \b x_1^-$ with $\a>0$, $\b \ge 0$ and
$\a^2 - \b^2 = 2a$.  

\begin{conjecture}  \label{conj:onedimensional}
Suppose that $f\in C_0^\infty(\R)$ and $f \ge 0$.
Suppose that $u$ is a solution to $\Delta u = f(u)$ in 
$\R^n$ and there is a homeomorphism $\F:\R^n \to \R^n$ such
that $u(x)= U(\F(x))$ for some one dimensional solution $U$ to
the free boundary problem.  Then $u$ is a one-dimensional
solution, i.~e.,  after rigid motion, $u(x) = V(x_1)$ with $V''(s) = f(V(s))$,
at least for  $n\le 3$.
\end{conjecture}

\begin{question}  Is there a local version of Conjecture \ref{conj:onedimensional}?
Suppose that $f\in C_0^\infty(\R)$, $f\ge 0$.  Does the fact that $u$ 
solving $\Delta u = f(u)$ on the unit ball $\B\subset \R^n$ (say for $n=2$ or $n=3$)
is homeomorphic to a one-dimensional solution imply that it is nearly 
isometric to a one-dimensional solution in a smaller ball? More precisely,
suppose 
\[
u(x) = V_1(\langle \F_1(x), a_1\rangle)
\]
for some $V_1$ solving $V_1''(s) = f(V_1(s))$, $s\in \R$, some 
unit vector $a_1$ and some homeomorphism $\F_1$ from the closed unit
ball to itself.  Does there exist $\d>0$ so that 
on $\B_\delta$, 
\[u(x) = V_2(\langle \F_2(x), a_2 \rangle)
\]
for some one dimensional solution $V_2$, some
unit vector $a_2$, and some mapping $\F_2$ that is an isometry 
up to distortion of 1 percent?  
\end{question}  

Lastly, we mention some questions related to the work of 
Colding and Minicozzi on the Calabi-Yau conjecture \cite{CM-CalabiYau}.
Their main theorem says that the only complete, embedded  minimal
surfaces in $\R^3$ with finite topology contained in a half space are
planes.  One analogous statement for free boundary problems
is a theorem of Alt and Caffarelli saying that a classical solution
entire solution to the one-phase free boundary problem in $\R^3$,
whose positive phase is contained in a half space must be a half space
solution.   A  statement of the same sort for solutions
to semilinear equations is as follows.

\begin{conjecture}  Suppose that $f\in C_0^\infty(\R)$ is such that $f\ge 0$.
Suppose that $u$ solves $\Delta u = f(u)$ in $\R^3$ and $u$ has
finitely many critical points.  If $u(x) >0$ implies
$x_1>0$, then $u$ is a one-dimensional solution in the sense of Conjecture \ref{conj:onedimensional}. 
\end{conjecture}
\noi Similar half-space rigidity results were obtained in \cite{farina2010flattening, RosSic, wang2015serrin, RosRuizSic2017rigidity} in the context of the Berestycki-Caffarelli-Nirenberg Conjecture \cite{BCN}, and in \cite{LWW-AC-HalfSpace} for the Allen-Cahn equation in low dimension.

One of the main steps in the proof by Colding and Minicozzi of
the Calabi-Yau conjecture is the proof that embedded, complete minimal
surfaces with finite topology in $\R^3$ are properly embedded.  The
properness means that on every bounded subset, the intrinsic distance
on the surface is comparable to the straight line distance.   
Part of
the finite topology hypothesis is the assumption that the manifold has 
only finitely many ends.   Indeed, the comparability fails to
hold uniformly for the family of  helicoids.  Although there
is no family that corresponds to helicoids in the free boundary setting,
the double hairpin family in dimension 2 and the product of a line with a double
hairpin in dimension 3 also fails to satisfy the comparability of intrinsic
distance with straight line distance uniformly.  In both of these
examples the blow down limit is the wedge solution $W_1$, which
has disconnected positive phase.   We expect that this is the
only type of solution for which intrinsic and extrinsic distance are
not comparable.   

To make a concrete conjecture, we 
change our focus from the free boundary to its positive
phase.  Our statement says roughly that
the purely topological property of connectivity of the positive phase at infinity 
implies a uniform, quantitative connectivity of the positive positive phase 
at all scales.

\begin{conjecture}  Suppose that $f\in C_0^\infty(\R)$ is such that $f\ge 0$
and $f(s) = 0$ for all $s\le 0$.  Suppose that $u$ solves $\Delta u = f(u)$ in $\R^3$, and $u$ has a blow down limit $\disp U(x) = \lim_{\e\to 0^+} \e \, u(x/\e)$ whose  positive phase $\{x\in \R^3: U(x) >0\}$ is connected.  Then the minimum
distance between points by paths in $\{x\in \R^3: u(x)>0\}$ is comparable 
to the straight line distance between the points with a constant depending
only on $U$ or perhaps even an absolute constant. 
\end{conjecture}

\bibliography{caff70_Bib}
\end{document}

%% file: FB_AllCases.pdf_tex
\begingroup%
  \makeatletter%
  \providecommand\color[2][]{%
    \errmessage{(Inkscape) Color is used for the text in Inkscape, but the package 'color.sty' is not loaded}%
    \renewcommand\color[2][]{}%
  }%
  \providecommand\transparent[1]{%
    \errmessage{(Inkscape) Transparency is used (non-zero) for the text in Inkscape, but the package 'transparent.sty' is not loaded}%
    \renewcommand\transparent[1]{}%
  }%
  \providecommand\rotatebox[2]{#2}%
  \newcommand*\fsize{\dimexpr\f@size pt\relax}%
  \newcommand*\lineheight[1]{\fontsize{\fsize}{#1\fsize}\selectfont}%
  \ifx\svgwidth\undefined%
    \setlength{\unitlength}{497.26521981bp}%
    \ifx\svgscale\undefined%
      \relax%
    \else%
      \setlength{\unitlength}{\unitlength * \real{\svgscale}}%
    \fi%
  \else%
    \setlength{\unitlength}{\svgwidth}%
  \fi%
  \global\let\svgwidth\undefined%
  \global\let\svgscale\undefined%
  \makeatother%
  \begin{picture}(1,0.60510852)%
    \lineheight{1}%
    \setlength\tabcolsep{0pt}%
    \put(0,0){\includegraphics[width=\unitlength,page=1]{FB_AllCases.pdf}}%
  \end{picture}%
\endgroup%

%% file: Sec3Fig.pdf_tex
\begingroup%
  \makeatletter%
  \providecommand\color[2][]{%
    \errmessage{(Inkscape) Color is used for the text in Inkscape, but the package 'color.sty' is not loaded}%
    \renewcommand\color[2][]{}%
  }%
  \providecommand\transparent[1]{%
    \errmessage{(Inkscape) Transparency is used (non-zero) for the text in Inkscape, but the package 'transparent.sty' is not loaded}%
    \renewcommand\transparent[1]{}%
  }%
  \providecommand\rotatebox[2]{#2}%
  \newcommand*\fsize{\dimexpr\f@size pt\relax}%
  \newcommand*\lineheight[1]{\fontsize{\fsize}{#1\fsize}\selectfont}%
  \ifx\svgwidth\undefined%
    \setlength{\unitlength}{162.45360539bp}%
    \ifx\svgscale\undefined%
      \relax%
    \else%
      \setlength{\unitlength}{\unitlength * \real{\svgscale}}%
    \fi%
  \else%
    \setlength{\unitlength}{\svgwidth}%
  \fi%
  \global\let\svgwidth\undefined%
  \global\let\svgscale\undefined%
  \makeatother%
  \begin{picture}(1,1.16097451)%
    \lineheight{1}%
    \setlength\tabcolsep{0pt}%
    \put(0,0){\includegraphics[width=\unitlength,page=1]{Sec3Fig.pdf}}%
  \end{picture}%
\endgroup%

%% file: Hairpin.pdf_tex
\begingroup%
  \makeatletter%
  \providecommand\color[2][]{%
    \errmessage{(Inkscape) Color is used for the text in Inkscape, but the package 'color.sty' is not loaded}%
    \renewcommand\color[2][]{}%
  }%
  \providecommand\transparent[1]{%
    \errmessage{(Inkscape) Transparency is used (non-zero) for the text in Inkscape, but the package 'transparent.sty' is not loaded}%
    \renewcommand\transparent[1]{}%
  }%
  \providecommand\rotatebox[2]{#2}%
  \newcommand*\fsize{\dimexpr\f@size pt\relax}%
  \newcommand*\lineheight[1]{\fontsize{\fsize}{#1\fsize}\selectfont}%
  \ifx\svgwidth\undefined%
    \setlength{\unitlength}{313.52139648bp}%
    \ifx\svgscale\undefined%
      \relax%
    \else%
      \setlength{\unitlength}{\unitlength * \real{\svgscale}}%
    \fi%
  \else%
    \setlength{\unitlength}{\svgwidth}%
  \fi%
  \global\let\svgwidth\undefined%
  \global\let\svgscale\undefined%
  \makeatother%
  \begin{picture}(1,0.50658617)%
    \lineheight{1}%
    \setlength\tabcolsep{0pt}%
    \put(0,0){\includegraphics[width=\unitlength,page=1]{Hairpin.pdf}}%
  \end{picture}%
\endgroup%

%% file: Scherk.pdf_tex
\begingroup%
  \makeatletter%
  \providecommand\color[2][]{%
    \errmessage{(Inkscape) Color is used for the text in Inkscape, but the package 'color.sty' is not loaded}%
    \renewcommand\color[2][]{}%
  }%
  \providecommand\transparent[1]{%
    \errmessage{(Inkscape) Transparency is used (non-zero) for the text in Inkscape, but the package 'transparent.sty' is not loaded}%
    \renewcommand\transparent[1]{}%
  }%
  \providecommand\rotatebox[2]{#2}%
  \newcommand*\fsize{\dimexpr\f@size pt\relax}%
  \newcommand*\lineheight[1]{\fontsize{\fsize}{#1\fsize}\selectfont}%
  \ifx\svgwidth\undefined%
    \setlength{\unitlength}{413.61082143bp}%
    \ifx\svgscale\undefined%
      \relax%
    \else%
      \setlength{\unitlength}{\unitlength * \real{\svgscale}}%
    \fi%
  \else%
    \setlength{\unitlength}{\svgwidth}%
  \fi%
  \global\let\svgwidth\undefined%
  \global\let\svgscale\undefined%
  \makeatother%
  \begin{picture}(1,0.44500203)%
    \lineheight{1}%
    \setlength\tabcolsep{0pt}%
    \put(0,0){\includegraphics[width=\unitlength,page=1]{Scherk.pdf}}%
  \end{picture}%
\endgroup%